\newif\ifpictures
\picturestrue

\documentclass[12pt]{amsart}
\usepackage{amssymb}
\usepackage{amsmath}
\usepackage{amscd}
\usepackage[dvips]{graphicx}
\usepackage{book tabs}
\usepackage{caption,subcaption}

\usepackage{dsfont}

\usepackage{float}

\ifpictures
\usepackage{psfrag}
\fi

\numberwithin{equation}{section}
\newtheorem{thm}{Theorem}
\newtheorem{prop}[thm]{Proposition}
\newtheorem{lemma}[thm]{Lemma}
\newtheorem{cor}[thm]{Corollary}

\theoremstyle{definition}

\newtheorem{assumption}[thm]{Assumption}
\newtheorem{example}[thm]{Example}
\newtheorem{remark1}[thm]{Remark}
\newtheorem{openproblem1}[thm]{Open problem}
\newtheorem{definition}[thm]{Definition}

\newenvironment{ex}{\begin{example}\rm}{\hfill$\Box$\end{example}}
\newenvironment{rem}{\begin{remark1}\rm}{\end{remark1}}

\numberwithin{thm}{section}

\newcounter{FNC}[page]
\def\newfootnote#1{{\addtocounter{FNC}{2}$^\fnsymbol{FNC}$%
     \let\thefootnote\relax\footnotetext{$^\fnsymbol{FNC}$#1}}}

\newcommand{\N}{\mathbb{N}}

\newcommand{\R}{\mathbb{R}}

\newcommand{\B}{\mathbb{B}}

\newcommand{\cal}{\mathcal}
\newcommand{\T}{\mathbb{T}}

\newcommand{\HH}{\mathcal{H}}

\newcommand{\mom}{\mathrm{mom}}
\newcommand{\sos}{\mathrm{sos}}

\newcommand{\widephi}{\widehat{\Phi}}
\newcommand{\widea}{\widehat{A}}

\newcommand{\ca}{$C^*$-algebra}

\DeclareMathOperator{\lin}{lin}

\DeclareMathOperator{\inter}{int}

\DeclareMathOperator{\linspan}{span}
\DeclareMathOperator{\diag}{diag}

\DeclareMathOperator{\tr}{tr}

\newcommand{\sym}{\mathcal{S}}
\newcommand{\psd}{\mathcal{S}^{+}}

\usepackage{xcolor}

\title[A semidefinite hierarchy for containment of spectrahedra]{A
semidefinite hierarchy for containment of spectrahedra}

\author{Kai Kellner}

\author{Thorsten Theobald}

\author{Christian Trabandt}

\address{Goethe-Universit\"at, FB 12 -- Institut f\"ur Mathematik,
Postfach 11 19 32, D--60054 Frankfurt am Main, Germany}

\email{\{kellner,theobald,trabandt\}@math.uni-frankfurt.de}

\usepackage{hyperref}

\begin{document}
\begin{abstract} 
A spectrahedron is the positivity region of a linear matrix pencil and thus
the feasible set of a semidefinite program.
We propose and study a hierarchy of sufficient semidefinite conditions to
certify the containment of a spectrahedron in another one.
This approach comes from applying a moment relaxation to a suitable polynomial
optimization formulation.
The hierarchical criterion is stronger than a solitary semidefinite criterion 
discussed earlier by Helton, Klep, and McCullough as well as by the authors.
Moreover, several exactness results for the solitary criterion can be brought
forward to the hierarchical approach.

The hierarchy also applies to the (equivalent) question of checking whether a
map between matrix (sub-)spaces is positive. In this context, the solitary
criterion checks whether the map is completely positive, and thus our results
provide a hierarchy between positivity and complete positivity.
\end{abstract}

\maketitle

\section{Introduction}

Containment problems of convex sets belong to the classical problems in convex
geometry (see, e.g., Gritzmann and Klee for the containment of polytopes
\cite{gritzmann-klee-containment-survey}, Freund and Orlin for containment
problems of balls in balls \cite{freund-orlin-85}, or Mangasarian for
containment of convex sets in reverse-convex sets \cite{Mangasarian2002}). 

In this paper, we consider the containment problem for spectrahedra using the
following common notation.
Let $ \sym_k $ be the set of real symmetric $k\times k$-matrices, $\sym_k^+$
be the set of positive semidefinite $k\times k$-matrices, and $\sym_k[x]$ be
the set of symmetric $k\times k$-matrices with polynomial entries in 
$ x = ( x_1, \ldots ,x_n ) $.
For $ A_0, \ldots, A_n \in \sym_k $, denote by $ A(x) $ the \emph{linear
(matrix) pencil} $A(x) \ = \ A_0 + x_1 A_1 + \cdots + x_n A_n \in \sym_k[x]$.
The set
\begin{equation} \label{eq:spectrahedron1}
  S_A \ = \ \{x \in \R^n \, : \, A(x) \succeq 0 \}
\end{equation}
is called a \emph{spectrahedron}, where $A(x) \succeq 0$ denotes positive
semidefiniteness of the matrix $A(x)$. Our work is intrinsically motivated by
the fact that spectrahedra have become an important class of non-polyhedral
sets due to the availability of fast semidefinite programming solvers. See
\cite{bpt-2013,Boyd2004,Ramana1995,helton-nie-2010,Pataki2000} for general
background on spectrahedra, and their significance in optimization and convex
algebraic geometry. Spectrahedra can be used to represent observables in
quantum information theory~\cite{weis-2011}. From an application point of
view, interest in non-polyhedral, and thus particularly semidefinite, set
containment is stimulated by non-polyhedral knowledge based data
classification (see \cite{jeyakumar-2003,Mangasarian2002}, for semidefinite
classifiers see \cite{jow-2006}).

Given two linear pencils $ A(x) \in\sym_k[x] $ and $ B(x) \in\sym_l[x] $, 
the containment problem for spectrahedra is to decide whether $S_A \subseteq
S_B$. This problem is co-NP-hard \cite{ben-tal-nemirovski-2002,Kellner2012}. 
The study of algorithmic approaches and relaxations has been initiated by
Ben-Tal and Nemirovski~\cite{ben-tal-nemirovski-2002} who investigated the
case where $S_A$ is a cube (``matrix cube problem''). 
Helton, Klep, and McCullough~\cite{Helton2010} studied containment problems of
matricial positivity domains (which live in infinite-dimensional spaces), and
from this they derive a \emph{semidefinite} sufficient criterion for deciding
containment of spectrahedra. In~\cite{Kellner2012}, the authors of the present
paper provided a streamlined presentation of the results on spectrahedral
containment in~\cite{Helton2010} and showed that in several cases the
sufficient criterion is exact.

 From an operator algebra point of view (such as 
in~\cite{Helton2010,Klep2011}), this semidefinite feasibility condition states
that a natural linear map $\Phi_{AB}$ between the subspaces 
$\mathcal{A} = \linspan(A_0, \ldots, A_n)$ and 
$\mathcal{B} = \linspan(B_0, \ldots, B_n)$ is \emph{completely positive} (as
defined in Section~\ref{sec:pos_maps}). These maps also appear in the context
of Positivstellens\"atze in non-commuting variables; see \cite{Helton2012}.
Building upon these results, in the current paper we go one step further, 
presenting a hierarchy of monotone improving sufficient semidefinite
optimization problems to decide the containment question. 

Our point of departure is to formulate the containment problem in terms of
polynomial matrix inequalities (PMI). We use common relaxation techniques (by
Kojima \cite{Kojima2003}, Hol and Scherer~\cite{Hol} as well as Henrion and
Lasserre \cite{Henrion2006}) to derive a (sufficient) semidefinite hierarchy
for the containment problem. The semidefinite hierarchy provides a much more
comprehensive approach towards the containment problems than the
aforementioned sufficient criterion (see Theorem~\ref{thm:implications1}).
We also discuss a variant of the semidefinite hierarchy which avoids
additional variables (see Section~\ref{sec:robust}).

\medskip

\noindent
{\bf Main contributions.}
 1. Based on polynomial matrix inequalities, we provide a hierarchy of
  sufficient semidefinite criteria for the containment problem and prove that
  the sequence of optimal values converges to the optimal value of the
  underlying polynomial optimization problem; see
  Theorem~\ref{thm:convergence}.

2. Any relaxation step of the hierarchy yields a sufficient criterion for the
  containment problem. We prove that each of these sufficient criteria is at
  least as powerful as the one in~\cite{Helton2010,Kellner2012}, in the sense
  that whenever the criterion of~\cite{Helton2010,Kellner2012} is satisfied,
  then also the criterion from any of the relaxation steps of the hierarchy
  is satisfied; see Theorem~\ref{thm:implications1}. In particular, this
  already holds for the criterion coming from the initial relaxation step.
  This allows to carry all exactness results from~\cite{Kellner2012} forward
  to our new hierarchical approach; see Corollaries~\ref{co:exactness}
  and~\ref{co:scaling}.

3. Application of the hierarchy to the problem of deciding whether a linear
  map between matrix subspaces is positive gives a monotone semidefinite
  hierarchy of sufficient criteria for this problem.

4. We demonstrate the effectiveness of the approach by providing
   numerical results for several containment problems and radii computations.

\medskip

We remark that the containment question is intimately linked to the
computation of inner and outer radii of convex sets. (See Gritzmann and
Klee~\cite{Gritzmann1992,gritzmann-klee-93} for the polytope case). Moreover,
Bhardwaj, Rostalski, and Sanyal~\cite{Bhardwaj2011} study the related question
of whether a spectrahedron is a polytope. In~\cite{Gouveia2013}, Gouveia,
Robinson, and Thomas reduced the question of computing the positive
semidefinite rank of nonnegative matrices to a containment problem involving
projections of spectrahedra.

The paper is structured as follows. In Section~\ref{sec:prelim}, we collect
some notations and concepts on spectrahedra and polynomial matrix
inequalities. The semidefinite hierarchy, as well as a variant avoiding
additional variables, is introduced in Section~\ref{se:hierarchy}. 
In Section~\ref{se:positivity}, we connect the hierarchy to (complete)
positivity of operators and the sufficient semidefinite criteria from
\cite{Helton2010,Kellner2012}.
Section~\ref{sec:application} discusses applications of radii computations
and provides numerical results.

\section{Preliminaries}
\label{sec:prelim}

Throughout, general matrix polynomials will be denoted by $G(x)\in\sym_k[x]$,
while linear matrix pencils will usually be denoted by $A(x)\in \sym_k[x]$ or
$B(x)\in \sym_l[x]$. Let $I_n$ abbreviate the $n \times n$ identity matrix,
and let $E_{ij}$ denote the matrix with a one in position $(i,j)$ and zeros 
elsewhere. By $\B_r(p)$, we denote the (closed) Euclidean ball with center $p$
and radius $r>0$.

\subsection{Spectrahedra and semidefinite programming}
\label{sec:spec}

Given a linear pencil 
\begin{equation} \label{eq:linearpencil}
  A(x) \ = \ A_0 + \sum_{p=1}^n x_p A_p \in \mathcal{S}_k[x] \quad
  \text{ with } A_p = (a^p_{ij}) \, , \quad 0 \le p \le n \, ,
\end{equation}
the spectrahedron $S_A = \{x \in \R^n \, : \, A(x) \succeq 0\}$ contains the
origin if and only if $ A_0 $ is positive semidefinite. Since the class of
spectrahedra is closed under translation, this can always be achieved (assuming
that $S_A$ is nonempty). Indeed, there exists a point $ x'\in\R^n $ such that 
$ A(x')\succeq 0 $ if and only if the origin is contained in the set 
$ \left\{ x\in\R^n\ :\ A(x+x') \succeq 0 \right\} $. In particular, the
constant term in the linear pencil $A'(x) = A(x+x')$ is positive semidefinite. 

The equivalence between positive definiteness of $A_0$ and the origin being an
interior point is not true. Moreover, in general, the interior of $S_A$ does
not coincide with the positive definiteness region of the pencil. However, if
the spectrahedron $S_A$ has nonempty interior (or, equivalently, $S_A$ is
full-dimensional), then there exists a \textit{reduced} linear pencil that is
positive definite exactly on the interior of $S_A$. 

\begin{prop}[{{\cite[Corollary 5]{Ramana1995}}}] \label{prop:reduced_pencil}
Let $ S_A = \{ x\in\R^n\ : \ A(x) \succeq 0 \} $ be full-dimensional and let
$N$ be the intersection of the nullspaces of $A_i,\ i=0,\ldots,n$. If $V$ is a
basis of the orthogonal complement of $N$, then 
$ S_A = \{ x\in\R^n\ : \ V^T A(x) V \succeq 0 \} $ and the interior of
$S_A$ is $ \inter (S_A )= \{ x \in \R^n \ :\ V^T A(x) V \succ 0 \} $.
\end{prop}

Furthermore, the spectrahedron $S_A$ contains the origin in its interior if
and only if there is a linear pencil $A'(x)$ with the same positivity
domain such that $A'_0 = I_k$; see \cite{helton-vinnikov-2007}.
To simplify notation, we sometimes assume that $A(x)$ is of this form and
refer to it as a \textit{monic} linear pencil, i.e., $A_0 = I_k$.

In addition, we occasionally assume the matrices $A_1, \ldots, A_n$ to be
linearly independent. This assumption is not too restrictive. 
In order to see this, denote by $ \tilde{A}(x) = A(x) - A_0 $  the pure-linear
part of the linear pencil $A(x)$.
Recall the well-known fact that the \emph{lineality space} $L_A$ of a
spectrahedron $S_A$, i.e.,\ the largest linear subspace contained in $S_A$, is
the set $ L_A = \{ x\in\R^n\ :\ \tilde{A}(x) = 0 \} $; 
see~\cite[Lemma 3]{Ramana1995}. 
Obviously, if the coefficient matrices $A_1,\ldots,A_n$ are linearly
independent, then the lineality space is zero-dimensional, i.e., 
$ L_A = \{ 0 \} $. 
In particular, this is the case whenever the spectrahedron $S_A$ is bounded 
(and $A_0 \succeq 0$); see~\cite[Proposition 2.6]{Helton2010}. Conversely, if
there are linear dependencies in the coefficient matrices, then we can simply
reduce the containment problem to lower dimensions.

\begin{prop} \label{prop:lineality}
Let $ A(x) \in\sym_k[x] $ and $ B(x) \in\sym_l[x] $ be linear pencils such that
$S_A$ is non-empty.
\begin{enumerate}
  \item 
  $L_A = \{ 0 \}$ if and only if $A_1,\ldots,A_n$ are linearly independent.
  \item 
  If $ S_A \subseteq S_B $, then $ L_A \subseteq L_B $.
  \item
  If $ L_A \subseteq L_B $, then $ S_A \subseteq S_B $ holds 
  if and only if $ S_{A'} \subseteq S_{B'} $ holds, 
  where $S_{A'} = S_A \cap L_A^{\bot} $
  and $S_{B'} = S_B \cap L_A^{\bot}$.
\end{enumerate}
\end{prop}

To prove the Proposition, we need a result concerning the lineality space of 
a closed convex set.

\begin{lemma} \cite[Theorem 2.5.8]{Webster1994} \label{lem:webster}
Let $S$ be a non-empty closed convex set in $\R^n$ with lineality space $L$. 
Then $ S = L + (S \cap L^{\bot})$ and the convex set $S \cap L^{\bot}$
contains no lines.
\end{lemma}

\begin{proof} (of Proposition~\ref{prop:lineality})

\noindent To (1): This follows directly from  
$ L_A = \{ x\in\R^n\ :\ \tilde{A}(x) = 0 \} $
and the definition of linear independence.

\noindent To (2): 
If $ L_A = \{ 0 \} $, then $ L_A \subseteq L_B $ is obviously true. 
Therefore, assume $L_A \neq \{ 0 \}$.
Let $ \bar{x} \in S_A \subseteq S_B $ and $ 0 \neq x \in L_A $. 
As above, denote by $\tilde{B}(x) = B(x) - B_0 = \sum_{p=1}^n x_p B_p $ the 
pure-linear part of $B(x)$. 
Then $ A(\bar{x}+tx)\succeq 0$ for all $t\in\R$ and hence 
$ B(\bar{x}) \pm t \tilde{B}(x) = B( \bar{x}\pm tx) \succeq 0 $ for all
$t\in\R$.
Consequently, $\pm \tilde{B}(x) \succeq 0$, i.e.,\ $\tilde{B}(x) = 0$. Thus the
linear subspace $\linspan{x}$ is contained in $L_B$.
Since $0\neq x\in L_A$ was arbitrary and $L_B$ is a linear subspace, we have 
$L_A \subseteq L_B$.

\noindent To (3):
Assume first $S_A \subseteq S_B$ holds. Then 
$S_{A'} = S_A \cap L_A^{\bot} \subseteq S_B \cap L_A^{\bot} = S_{B'}$.
For the converse, note that $S_A = L_A + S_{A'}$. Let $x\in S_{A}$. Then $x =
x_1 + x_2$ with $x_1 \in L_A$ and $x_2 \in S_{A'}$.
Since $x_1 \in L_A \subseteq L_B$ and
$x_2 \in S_{A'} \subseteq S_{B'} \subseteq S_B$, we have 
$ x \in L_B + S_B = S_B $.
\end{proof}

A (linear) \emph{semidefinite program} (SDP) is an optimization problem where
one optimizes a linear objective function $c^T x$ over a spectrahedron,
$ \inf\ \{ c^Tx \, : \, A(x) \succeq 0 \} $. 
A \emph{semidefinite feasibility problem} (SDFP) is the decision problem of
deciding whether for a given linear pencil $A(x)$ the spectrahedron $S_A$ is
nonempty. 
While SDPs (with rational input data) can be approximated in polynomial time
(see \cite{de-klerk-book}), the complexity of SDFP is open. The best known
results are contained in~\cite{ramana1997b}.
In practice, however, SDFPs can be solved efficiently by semidefinite
programming.

\subsection{Polynomial matrix inequalities}
\label{sec:pmi}

Problems involving a polynomial objective function and positive semidefinite
constraints on matrix polynomials are called \emph{polynomial matrix
inequality} (PMI) problems and can be written in the following standard form. 

\begin{align}
\begin{split} \label{opt:PMI}
  \inf\ &\ f(x) \\
  \text{s.t.}\ &\ G(x) \succeq 0,
\end{split} 
\end{align}
where $f(x) \in \R[x]$ and $G(x) \in \sym_k[x]$, not necessarily linear, 
for $ x = (x_1,\ldots,x_n)$. 

Hol and Scherer~\cite{Hol}, and Kojima~\cite{Kojima2003} introduced sums of
squares relaxations for PMIs, leading to semidefinite programming relaxations
of the original problem. Here we focus mainly on the dual viewpoint of moment
relaxations, as exhibited by Henrion and Lasserre~\cite{Henrion2006}. As in
Lasserre's moment method for polynomial optimization~\cite{Lasserre2001}, the
basic idea is to linearize all polynomials by introducing a new variable for
each monomial. The relations among the monomials give semidefinite conditions
on the moment matrices.

As discussed in~\cite{Henrion2006}, directly linearizing the positive
semidefiniteness condition~\eqref{opt:PMI} can lead to relaxations that use a
relatively small number of variables. To formalize this, let $[x]$ be the
monomial basis of $\R[x]$ and let $ y = \{ y_{\alpha} \}_{ \alpha \in \N^n } $
be a real-valued sequence indexed in the basis $[x]$. A polynomial $p(x) \in
\R[x]$ can be identified by its vector of coefficients $\vec{p}$ in the basis
$[x]$. By $[x]_t$ we denote the truncated basis containing only monomials of
degree at most $t$. For the linearization operation, consider the operator
$L_{y}$ defined by the linear mapping 
$ p \mapsto L_{y}(p) = \langle \vec{p}, y \rangle $. 

Let $M( y )$ be the moment matrix defined by 
$
  [ M ( y ) ]_{\alpha, \beta} = L_{y}( [[x] [x]^T ]_{\alpha, \beta}) 
  = y_{\alpha + \beta}
$.
$M_t(y)$ denotes the truncated moment matrix that contains only entries 
$[M(y)]_{\alpha,\beta}$ with $ |\alpha|, |\beta| \leq t $.

The positive semidefiniteness constraint on a matrix polynomial
$G(x)\in\sym_k[x]$ can be modelled by so called \emph{localizing matrices}
(which for $1 \times 1$-matrices specialize to the usual localizing matrices
within Lasserre's relaxation for polynomial optimization \cite{Lasserre2001}). 
The truncated localizing matrix $M_t(Gy)$ is the block matrix obtained by 
\[ 
  [M_t(Gy)]_{\alpha,\beta}=L_{y}([[x]_t [x]_t^T\otimes G(x)]_{\alpha,\beta}) .
\]
We write $ M_t ( G y ) = L_{y} ( [x]_t [x]_t^T \otimes G(x) ) $ for short. Let
$d_G$ be the highest degree of a polynomial appearing in $G(x)$. With this
notation only linearization variables coming from monomials of degree at most
$2t+d_G$ appear in $M_t(Gy)$.

We arrive at the following hierarchy of semidefinite relaxations for the
polynomial optimization problem \eqref{opt:PMI}, 
\begin{align}
\begin{split} \label{opt:PMI-hierarchy}
  f_{\mom}^{(t)}\ =\ \inf\ &\ L_{y} ( f(x) ) \\
  \text{s.t.}\ &\ M_t(y) \succeq 0 \\
  \ &\ M_{t- \lceil d_G / 2 \rceil}(Gy) \succeq 0.
\end{split}
\end{align}
We only use the monomial basis of degree up to $ t-\lceil d_G / 2 \rceil $ in
the last constraint so that only moments coming from variables of degree $2t$
or lower appear in the whole optimization problem. Note that 
$ t = \lceil \max \{ d_G , d_f \} / 2 \rceil $ is the smallest possible
relaxation order, since for smaller $t$ there are unconstrained variables in
the objective or the truncated localizing matrix is undefined. We call 
$t = \lceil\max\{d_G,d_f\} / 2\rceil$ the \emph{initial relaxation order}.

The optimal value of the hierarchy \eqref{opt:PMI-hierarchy} converges under
mild assumptions to the optimal value of the original problem \eqref{opt:PMI}.
To make this statement precise, we call a matrix polynomial $S(x)\in\sym_k[x]$
a \emph{sum of squares} (or \emph{sos-matrix}) if it has a decomposition 
$S(x) = U(x) U(x)^T$ with $ U(x)\in\R^{k\times m}[x] $ for some positive
integer $m$. For $ k=1 $, $S(x)$ is called \emph{sos-polynomial}.

\begin{prop}{ \cite[Theorem 2.2]{Henrion2006}, see also \cite[Theorem 1]{Hol}.}
  \label{prop:convergence_las}
Let $ G(x) \in\sym_k[x] $. Assume there exists a polynomial 
$ p(x) = s(x) + \langle S(x),G(x)\rangle $ for some sos-polynomial $s(x) \in
\R[x]$ and some sos-matrix $S(x)\in \sym_k[x]$, such that the level set 
$ \{ x\in\mathbb{R}^n \, :\, p(x) \geq 0 \} $ is compact. Then 
$f_{\mom}^{(t)}\uparrow f^*$ as $t \rightarrow \infty$ in the semidefinite
hierarchy \eqref{opt:PMI-hierarchy}.
\end{prop}

\section{A (sufficient) semidefinite hierarchy\label{se:hierarchy}}

Let $ A(x)\in\sym_k[x] $ and $ B(x)\in\sym_l[x]$ be two linear pencils. In
this section, we provide an optimization formulation to decide the question 
of whether the spectrahedron $S_A$ is contained in $S_B$. Using a PMI
formulation of the containment problem, we first deduce a sufficient
semidefinite hierarchy and prove the convergence of the hierarchy
(Theorem~\ref{thm:convergence}). Afterwards, we state a second, in fact highly
related, approach based on a quantified semidefinite program; see
Subsection~\ref{sec:robust}.

\subsection{An optimization approach to decide containment of spectrahedra}
\label{sec:motivation}

Clearly, $S_A$ is contained in $S_B$ if and only if $A(x)\succeq 0$ implies
the positive semidefiniteness of $B(x)$. By definition, $ B(x)\succeq 0 $ for
arbitrary but fixed $ x\in\R^n $ is equivalent to the nonnegativity of the
polynomial $ z^T B(x) z $ in the variables $z=(z_1,\ldots,z_l)$. Thus, $S_A$
is contained in $S_B$ if and only if the infimum $\mu$ of the degree 3
polynomial $z^T B(x)z$ in $(x,z)$ over the spectrahedron $S_A\times\R^l$ is
nonnegative. Imposing a normalization condition on $z$, we arrive at the
following formulation.

\begin{prop} \label{prop:containment=psdp}
Let $ A(x)\in\sym_k[x] $ and $ B(x)\in\sym_l[x]$ be linear pencils with 
$S_A\neq\emptyset$, and let $ g_r(z) = z^T z-r^2,\ g^R(z) = R^2-z^T z $ for
arbitrary but fixed $0 < r\leq R$. For the polynomial optimization problem
\begin{align}
\begin{split} \label{eq:contain_poly}
  \mu \ =\ \inf\ &\ z^T B(x) z \\
  \mathrm{s.t.}\ &\ G_A(x,z) := \diag(A(x),g_r(z),g^R(z)) \succeq 0 
\end{split}
\end{align}
the following implications are true,
\begin{align*}
  \mu > 0  \ \Rightarrow \ & S_A \subseteq \inter S_B, \\
  \mu  = 0 \ \Rightarrow \ & S_A \subseteq S_B, \\
  \mu < 0  \ \Leftrightarrow \ & \exists x \in S_A: B(x) \nsucceq 0.
\end{align*}
If the pencil $B(x)$ is reduced in the sense of
Proposition~\ref{prop:reduced_pencil}, $\mu  = 0$ implies that the
spectrahedra touch at the boundary.
\end{prop}

A natural choice of the parameters $r$ and $R$ is to set both to $1$. In this
case, the optimal value of the optimization problem equals the smallest
eigenvalue of any matrix in the set $\{B(x) : x \in S_A\}$. Other choices
result in an optimal value that is scaled by $R^2$ in the case $\mu < 0$ and
by $r^2$ in the case $\mu > 0$. As our numerical computations in
Section~\ref{sec:application} show, the problem, or, more precisely, its
relaxation defined in Section~\ref{sec:derivation} is numerically
ill-conditioned if we chose $r = R$ and becomes more tractable for $r < R$.

In applications, it is advisable to use reduced pencils. The reduced pencil can
be computed by the methods in~\cite{Ramana1995} and makes the numerical
computations described below better conditioned. Not only do we expect a
strictly positive objective value whenever $S_A \subseteq\inter S_B$, the
reduced pencil is also of smaller size.

\begin{proof}[Proof (of Proposition~\ref{prop:containment=psdp})]
Denote by 
$ \T = \T_{r,R}(0) = \{ z\in\R^l\ :\ r^2 \leq z^T z \leq R^2 \} $ 
the annulus defined by the constraints $ g_r(z) \geq 0,\ g^R(z) \geq 0 $. 

We first observe that the existence of an $x \in S_A$ and $z \in \R^l $ with 
$ z^T B(x) z < 0 $ implies the existence of a point 
$ z':= R\cdot\frac{z}{\|z\|} \in \T $ with $ \| z' \| = R $ and 
$ z'^T B(x) z' < 0 $, and thus $(x,z')$ lies in the product of the
spectrahedron $S_A$ and the annulus $\T$.

If $\mu \ge 0$, then clearly $S_A \subseteq S_B$. To deduce the case $\mu > 0$,
observe that the boundary $\partial S_B$ of $S_B$ is contained in the set
\[
  \{x \in \R^n \, : \, B(x) \succeq 0,\ 
  z^T B(x) z = 0 \text{ for some } z \in \T \} \, .
\]
Hence, if the boundaries of $S_A$ and $S_B$ contain a
common point $\bar{x}$, then there exists some $\bar{z}$ such that the
objective value of $(\bar{x},\bar{z})$ is zero.
\end{proof}

\subsection{Derivation of the hierarchy using moment relaxation methods}
\label{sec:derivation}
 
Using the framework of moment relaxations for PMIs introduced in
Section~\ref{sec:pmi}, we consider the following semidefinite hierarchy as a
relaxation to problem~\eqref{eq:contain_poly}, providing a semidefinite
hierarchy for the containment question. Let $y$ be a real-valued sequence
indexed by $[x,z]$, the monomial basis of 
$\R[x,z] = \R[x_1,\ldots,x_n,z_1,\ldots,z_l]$. For $t\geq 2$, we obtain the
$t$-th relaxation of the polynomial optimization
problem~\eqref{eq:contain_poly}
\begin{align}
\begin{split} \label{eq:contain_lasserre}
  \mu_{\mom}(t)\ =\ \inf\ &\ L_{y} ( z^T B(x) z ) \\
  \text{s.t.}\ &\ M_t(y) \succeq 0 \\
  \ &\ M_{t-1}(G_Ay) \succeq 0 .
\end{split}
\end{align} 
As described in Subsection~\ref{sec:pmi}, we only use the monomial basis of 
degree up to $t-1$ in the last constraint so that only moments coming from
variables of degree $2t$ or lower appear in the whole optimization problem.
Note that $t=2$ is the initial relaxation order, as defined in
Section~\ref{sec:pmi}. By increasing $t$, additional constraints are added,
which implies the following corollary.

\begin{cor} \label{co:inclusion1}
The sequence $\mu_{\mom}(t)$ for $t \ge 2$ is monotone non-decreasing.
If for some $t^*$ the condition $\mu_{\mom}(t^*) \ge 0$ is satisfied, 
then $S_A \subseteq S_B$.
\end{cor}

That is, for any $t$, the condition $ \mu_{\mom}(t) \ge 0 $ provides a
sufficient criterion for the containment $S_A \subseteq S_B$. In the case when
the inner spectrahedron $S_A$ is bounded, the sequence of relaxations is not
only monotone non-decreasing, but also converges to the optimal value of the
original polynomial optimization problem~\eqref{eq:contain_poly}, as the next
theorem shows.

\begin{thm} \label{thm:convergence}
Let $A(x)\in\sym_k[x]$ be a linear pencil such that the spectrahedron $S_A$ is
bounded. Then the optimal value of the moment
relaxation~\eqref{eq:contain_lasserre} converges from below to the optimal
value of the polynomial optimization problem~\eqref{eq:contain_poly}, i.e.,
$ \mu_{\mom}(t) \uparrow \mu$ as $t\rightarrow\infty $.
\end{thm}

\begin{proof}
By Proposition~\ref{prop:convergence_las}, it suffices to show that there
exists an sos-polynomial $ s(x,z)\in\R[x,z] $ and an sos-matrix
$ S(x,z)\in\sym_{k+2}[x,z] $ defining a polynomial 
$ p(x,z) = s(x,z) + \left\langle S(x,z), G_A(x,z) \right\rangle $ 
such that the level set $ \{ (x,z) \in \R^{n+l}\ :\ p(x,z) \geq 0 \} $
is compact. Define the quadratic module 
\[
  M_A = \left\{ t(x) + \left\langle A(x), T(x) \right\rangle\ :\ 
  t(x)\in\R[x] \text{ sos-polynomial, } T(x)\in\sym_k[x] \text{ sos-matrix} 
  \right\} .
\]
As shown in \cite[Corollary 2.2.6]{Klep2011}, the boundedness of $S_A$ is
equivalent to the fact that the quadratic module $ M_A $ is Archimedean, 
i.e., there exists a positive integer $N\in\N$ such that $ N - x^T x \in M_A $.
Thus, by the definition of the quadratic module $M_A$, there exists an
sos-polynomial $ t(x)\in\R[x] $ and an sos-matrix $T(x)\in\sym_k[x]$ such that
\[
  N - x^T x = t(x) + \left\langle T(x) , A(x) \right\rangle . 
\]
Define $s(x,z) = t(x)$ and $S(x,z) = \diag(T(x),0,1)$. 
Both have the sos-property. Indeed, if $ T(x) = U(x) U(x)^T $ is an
sos-decomposition of $T(x)$, then 
$ S(x,z) = \diag(T(x),0,1) = \diag(U(x),0,1) \diag(U(x)^T,0,1) $
is one of $S(x,z)$. We get 
\[
  p(x,z) = N - x^T x + R^2 - z^T z 
  = s(x,z) + \left\langle S(x,z) , G_A(x,z) \right\rangle . 
\]
Since this polynomial defines the ball of radius $N + R^2$ centered at the
origin, $\B_{N+R^2}(0) \subset \R^{n+l}$, the level set is compact.
\end{proof}

\begin{rem}
Computing a certificate $N$ from the proof of the theorem can again be done by
the polynomial semidefinite program~\eqref{eq:contain_poly} and its
relaxation~\eqref{eq:contain_lasserre}. We have a deeper look on this in
Section~\ref{sec:radii}. In fact, the program stated there computes the
circumradius of the spectrahedron $S_A$, if it is centrally symmetric with
respect to the origin.
\end{rem}

If the optimal value of the polynomial reformulation~\eqref{eq:contain_poly}
equals zero, it might lead to numerical issues in the
relaxation~\eqref{eq:contain_lasserre} as it requires the computation of an
exact value via semidefinite programming. 
From a geometric point of view this occurs only in somewhat degenerate cases:
the spectrahedra touch at the boundary or the determinantal variety of $B(x)$
intersects the interior of the spectrahedron $S_A$; if the pencil $B(x)$
is reduced in the sense of Proposition~\ref{prop:reduced_pencil}, 
the latter case is not possible.

\subsection{An alternative formulation}
\label{sec:robust}

A crucial point in the polynomial optimization
approach~\eqref{eq:contain_poly} is the introduction of additional variables
$ z = (z_1,\ldots,z_l) $ already in the original, unrelaxed polynomial
formulation (see Proposition~\ref{prop:containment=psdp}). An alternative
approach would be to start from the following quantified semidefinite program
without additional variables,
\begin{align}
\begin{split} \label{eq:robust}
  \mu \ =\ \sup\ &\ \lambda\\
  \mathrm{s.t.}\ &\ B(x) - \lambda I_l \succeq 0\ \forall x \in S_A \, .
\end{split}
\end{align}
By a result on robust polynomial semidefinite programming by Hol and
Scherer~\cite{Hol2006} this class of problems can be solved by an approach
based on sum-of-squares matrix polynomials, leading to a hierarchy of the form
\begin{align} \label{eq:contain_sos}
\begin{split}
  \lambda_{\sos}(t) = \sup\ &\ \lambda \\
  \mathrm{s.t.} \ &\ B(x) - \lambda I_l 
  - ( \langle S_{i,j}(x) , A(x) \rangle )_{i,j=1}^l \text{ sos-matrix} \\
  \ &\ S(x) = (S_{i,j}(x))_{i,j=1}^l \in \sym_{kl}[x] \text{ sos-matrix} .
\end{split}
\end{align}
where $S(x)$ has $l\times l$ blocks of size $k\times k$ with entries of degree
at most $2t\geq 0$. Using Theorem 1 and Corollary 1 from~\cite{Hol2006}, we
can state the subsequent convergence statement for the sos-relaxation. The
proof of this theorem is very similar to the one of
Theorem~\ref{thm:convergence}. 

\begin{thm} \label{thm:sos_convergence}
Let $ A(x)\in\sym_k[x] $ be a linear pencil such that the spectrahedron $S_A$ 
is bounded. Then the optimal value of the
sos-relaxation~\eqref{eq:contain_sos} converges from below to the optimal
value of the quantified semidefinite optimization problem~\eqref{eq:robust},
i.e., $ \lambda_{\sos}(t) \uparrow \mu $ as $ t\rightarrow\infty $.
\end{thm}

While in the quantified semidefinite program no additional variables 
$ z = (z_1,\ldots,z_l) $ are needed, the number of unknowns of the relaxation
grows not only in the number of variables $n$ and the relaxation order $t$,
i.e., half the degree of the entries in $S(x)$, but also in the size of both
the outer pencil $l$ and the inner pencil $k$. 
To be more precise, using the approach of Hol and Scherer, the number of
unknowns in the SDP coming from the sos-relaxation is generically 
\[
  1 + \frac{1}{2} \binom{n+t}{t} \cdot 
  \left[ k^2 l^2 \binom{n+t}{t} + l^2 \binom{n+t}{t} +kl+l \right] - ml(l+1) ,
\]
where $m$ denotes the number of affine equation constraints arising in the
sos-formulation; see \cite[Section 5]{Hol2006}. In our main approach there are
\[ 
  \frac{1}{2} \binom{n+l+t}{t}\left[\binom{n+l+t}{t}-1\right]
\]
variables. In certain situations with small $t$ (i.e., $t \in \{0,1\}$), the
sos-approach may lead to SDPs with a simpler structure than our main approach.
We study this in detail in Section~\ref{sec:application}.

\section[Positivity and containment]{Positivity of matrix maps 
and the hierarchy for containment\label{se:positivity}}

In this section, we first review the containment criterion based on complete
positivity of operators that was studied in \cite{Helton2010, Kellner2012}. We
then prove that the sufficient criteria coming from our hierarchy of
relaxations are at least as strong as the complete positivity criterion by
showing that feasibility of the complete positivity criterion implies 
$\mu \geq 0$ in the initial relaxation step of the semidefinite
hierarchy~\eqref{eq:contain_lasserre}. From this relation, we get that in some
cases already the initial relaxation step gives an exact answer to the
containment problem; see Corollaries~\ref{co:exactness} and~\ref{co:scaling}.

For the convenience of the reader, we first collect the relevant connections
between the containment problem and (complete) positivity of maps between
matrix spaces; see
Statements~\ref{prop:comp_pos_operator}--\ref{lem:containment=pos_maps}.
Theorem~\ref{thm:implications1} gives our main result concerning the
containment criterion from~\cite{Helton2010, Kellner2012} and the semidefinite
hierarchy.

\subsection{(Completely) positive maps} 
\label{sec:pos_maps}

Besides providing a (numerical) answer to the containment question, the
semidefinite hierarchy~\eqref{eq:contain_lasserre} is useful to detect
positivity of linear maps between (subspaces of) matrix spaces.

The concepts discussed in this subsection can be defined in a much more
general setting, using the language of operator theory. See,
e.g.,~\cite{Paulsen2003} for an introduction to positive and completely
positive maps on \ca s.

\begin{definition}
Given two linear subspaces $\cal{A}\subseteq\R^{k \times k}$ and
$\cal{B}\subseteq\R^{l \times l}$, a linear map $\Phi:\
\cal{A}\rightarrow\cal{B}$ is
called \emph{positive} if every positive semidefinite matrix in $\cal{A}$ is
mapped to a positive semidefinite matrix in $\cal{B}$, i.e.,
$ \Phi(\cal{A}\cap\psd_k) \subseteq \cal{B}\cap\psd_l $.

The map $\Phi$ is called \emph{$d$-positive} if the map 
$ 
  \Phi_d:\ \R^{d\times d}\otimes\cal{A} \rightarrow
  \R^{d\times d}\otimes\cal{B},\ M\otimes A\mapsto M\otimes\Phi(A)
$ 
is positive, i.e.\ 
$ (\Phi(A_{ij}))_{i,j=1}^d \in \cal{B}^{d\times d}\cap\psd_{dl} $ for
$ (A_{ij})_{i,j=1}^d \in\cal{A}^{d\times d}\cap\psd_{dk} $. 

Finally, $\Phi$ is called \emph{completely positive} if $\Phi_d$ is positive
for all positive integers $d$. 
\end{definition}

Naturally, every $d$-positive map is $e$-positive for all positive integers
$e\leq d$. Provided that $\mathcal{A}$ contains a positive definite matrix,
complete positivity of $\Phi$ is equivalent to $k$-positivity; 
see~\cite[Theorem 6.1]{Paulsen2003}. Interestingly, in this situation every
completely positive map does have a completely positive extension to the full
matrix space and can therefore be represented by a positive semidefinite
matrix. This is well known in the general setting of \ca s and persists in our
real setting.

\begin{prop}[{\cite[Theorem 6.2.]{Paulsen2003}}]\label{prop:comp_pos_operator}
Let $ \cal{A} \subseteq \R^{k \times k}$ be a linear subspace containing a
positive definite matrix, then each completely positive map 
$\Phi:\ \cal{A}\rightarrow\R^{l\times l} $ has an extension to a completely
positive map $ \tilde\Phi:\ \R^{k \times k} \rightarrow \R^{l\times l} $. 

Moreover, complete positivity of the map $\tilde\Phi$ is equivalent to
positive semidefiniteness of the matrix 
$ C = (C_{ij})_{i,j=1}^k 
  = \sum_{i,j=1}^k (E_{ij}\otimes\tilde\Phi(E_{ij}))\in\sym_{kl} $,
where $E_{ij}$ denotes the $k\times k$-matrix with 1 in position $(i,j)$ and
zeros elsewhere.
\end{prop}

A significant implication of Proposition~\ref{prop:comp_pos_operator} is the
following. 
Given a linear subspace $\mathcal{A}$ containing a positive definite matrix,
a linear map $\Phi:\ \mathcal{A}\rightarrow\R^{l\times l}$ is completely
positive if and only if at least one of all possible extensions of $\Phi$ to
the whole matrix space is completely positive. 
The set of extensions is determined by linear equations, fixing some (but not
all) of the entries in the matrix $C$. Testing the partially indeterminate 
matrix $C$ for a positive semidefinite extension is a semidefinite feasibility
problem (SDFP). Recall from the Preliminaries~\ref{sec:spec} that while the
computational complexity of solving SDFPs is open, in practice it can be done
efficiently by semidefinite programming. In~Section~\ref{sec:contain1} we
apply this to containment of spectrahedra.

Surprisingly, positive maps on subspaces do not always have a positive
extension to the full space; see, e.g.,~\cite[Example 3.16]{Stormer1986}. And
even if they do, characterizations of positive maps exist merely in low
dimensions and in the setting of hermitian matrix
algebras~\cite{Woronowicz1976b,Woronowicz1976}. The structure of positive maps
on higher dimensional spaces is not completely
understood~\cite{Skowronek2009,Stormer1963}. 

As we will see in this section, checking positivity of a map on subspaces is
equivalent to checking containment for spectrahedra. We can thus apply our
hierarchy for the containment question. 

\subsection{Equivalence of positive maps and containment}
\label{sec:pos-contain}
Given the linear pencils $A(x)\in\sym_k[x]$ and $B(x)\in\sym_l[x]$, we call
the linear pencil 
\begin{equation} \label{eq:extended_pencil}
  \widea = 1\oplus A(x) = 1\oplus A_0 + \sum_{p=1}^n x_p (0\oplus A_p) 
\end{equation}
the \emph{extended linear pencil} of $A(x)$, where $\oplus$ denotes the direct
sum of matrices. Define the corresponding linear subspaces
\begin{align*}
  \cal{A} &= \linspan(A_0, A_1, \ldots, A_n) \subseteq \sym_k, \\
  \widehat{\cal{A}} 
  &= \linspan(1\oplus A_0, 0\oplus A_1, \ldots, 0\oplus A_n) 
  \subseteq \sym_{k+1}, \text{ and}\\
  \cal{B}& = \linspan(B_0, B_1, \ldots, B_n) \subseteq \sym_l. 
\end{align*}
For linearly independent $A_1,\ldots,A_n$, let  
$ \widephi_{AB}: \widehat{\cal{A}} \rightarrow \cal{B} $ 
be the linear map defined by
\[
  \widephi_{AB} (1 \oplus A_0) = B_0 
  \quad \text{and} \quad \forall p \in \{1,\ldots ,n\}: \
  \widephi_{AB} (0 \oplus A_p) = B_p.
\]
Note that since every linear combination 
$ 0 = \lambda_0 (1\oplus A_0) + \sum_{p=1}^n \lambda_p (0\oplus A_p) $
for real $\lambda_0,\ldots,\lambda_n$ yields $\lambda_0 = 0$, it suffices to
assume the linear independence of the coefficient matrices $A_1,\ldots,A_n$
to ensure that $\widephi_{AB}$ is well-defined. To obtain linear independence,
the lineality space can be treated separately, as described in the
Preliminaries~\ref{sec:spec}. Note that the lineality space for the extended
pencil is the same as for the actual pencil.

If additionally, $A_0, A_1, \ldots, A_n$ are linearly independent, we can
retreat to the simpler map $ \Phi_{AB}: \cal{A}\rightarrow\cal{B}$ defined by 
\begin{align*}
  &\forall p \in \{0,\ldots ,n\}: \ \Phi_{AB}: A_p  \mapsto B_p.
\end{align*}

\begin{assumption}
Let $A_0, \ldots, A_n$ be linearly independent for statements concerning
$\Phi_{AB}$ and let $A_1, \ldots ,A_n$ be linearly independent for statements
concerning $\widehat\Phi_{AB}$. 
\end{assumption}

In \cite[Theorem 3.5]{Helton2010} the authors state the relationship between
$d$-positive maps and the question of containment of (bounded) matricial
positivity domains which for $d=1$ contains the case of spectrahedra. The
proof there is based on operator algebra. We give a more streamlined proof
concerning positive maps and spectrahedra.

\begin{prop} \label{prop:containment=maps}
Let $A(x)\in\sym_k[x]$  and $B(x)\in\sym_l[x]$ be linear pencils.
\begin{enumerate}
  \item 
  If $\Phi_{AB}$ or $\widephi_{AB}$ is positive, then $S_A\subseteq S_B$.
  \item 
  If $ S_A \neq\emptyset $, 
  then $ S_A \subseteq S_B $ implies $ \widephi_{AB} $ is positive.
  \item 
  If $ S_A \neq\emptyset $ and $S_A$ is bounded, 
  then $ S_A \subseteq S_B $ implies $ \Phi_{AB} $ is positive.
\end{enumerate}
\end{prop}

\begin{proof} To (1): 
Let $\Phi_{AB}$ be positive and let $x \in S_A$. Then $A(x) \succeq 0$. By the
positivity of $\Phi_{AB}$, we have $B(x) = \Phi_{AB}(A(x))\succeq 0$, and thus
$x\in S_B$. There is no difference in the proof if $\widephi_{AB}$ is positive.

\smallskip

\noindent To (2): 
First note that $ A(x)\succeq 0 $ if and only if the extended linear pencil
$ \widea(x) $ is positive semidefinite. Hence $ S_A = S_{\widea} $. Set 
$ \widea(x_0, x) := x_0 (1\oplus A_0) + \sum_{p=1}^n x_p (0\oplus A_p) $
and let $x_0 \in \R$ with $\widea(x_0,x)\in\widehat{\cal{A}}\cap \psd_{k+1}$.
Then $x_0 \geq 0$.

\paragraph{Case $x_0 > 0$.} 
Then $ \widea(1,x/x_0) = \frac{1}{x_0} \widea(x_0,x) \succeq 0 $.
Thus, $ \frac{x}{x_0}\in S_{\widea} = S_A \subseteq S_B $ and 
$ B(x_0,x) = x_0 B(1,x/x_0) \succeq 0 $. 
We get 
$ \widephi_{AB} (\widea(x_0,x)) = B(x_0,x) \in \cal{B}\cap\psd_l $.

\paragraph{Case $x_0 = 0$.} 
By assumption, $S_A$ is nonempty, i.e., there exists a point $\bar{x}\in S_A$.
Then $ A(0,x) \succeq 0 $ together with the positive semidefiniteness of 
$A(1,\bar{x}) $  (or, equivalently, $ \widea(1,\bar{x})$) implies 
$ \bar{x} + tx\in\ S_A \subseteq S_B $ for all $t\geq 0$. Thus $ x $ is a
point of the recession cone of $S_A$ which clearly is contained in the
recession cone of $S_B$. Consequently, 
$ \frac{1}{t} B(1,\bar{x}) + B(0,x) \succeq 0 $ for all $ t > 0 $. By
closedness of the cone of positive semidefinite matrices, we get 
$ B(0,x) \succeq 0 $. Hence
$ \widephi_{AB}(\widea(x_0,x)) = \widephi_{AB}(\widea(0,x)) = B(0,x)\succeq 0$.

\smallskip

\noindent To (3): 
Let $S_A\subseteq S_B$ and $A(x_0,x)=x_0 A_0+\sum_{p=1}^n x_p A_p \in\cal{A}$
be positive semidefinite. 

\paragraph{Case $x_0 \leq 0$.} 
Since $S_A$ is nonempty, there exists $ \bar{x}\in\R^n $ such that 
$ A(1,\bar{x}) \succeq 0 $, and hence 
\[
  A(0,x+|x_0| \bar{x}) = A(0,x) + A(0,|x_0|\bar{x}) 
  \succeq |x_0| A_0 + A(0,|x_0|\bar{x}) = |x_0|\cdot A(1,\bar{x}) \succeq 0 .
\]
For $ A(0,x+|x_0| \bar{x}) \neq 0 $, one has an improving ray of the
spectrahedron $S_A$, in contradiction to the boundedness of $S_A$.
For $ A(0,x+|x_0| \bar{x}) = 0 $, the linear independence of $A_0,\ldots,A_n$
implies $ x + |x_0| \bar{x} = 0 $. But then 
$ x_0 A(1,\bar{x}) = A(x_0,x)\succeq 0$ together with $x_0 \leq 0$ and 
$A(1,\bar{x}) \succeq 0 $ implies either $A(1,\bar{x})= 0$, in contradiction
to linear independence, or $(x_0,x) = 0$. Clearly, $\Phi_{AB}(0) = 0$. 

\paragraph{Case $x_0 > 0$.} Then $x/x_0\in S_A\subseteq S_B$. Thus,
$\Phi_{AB}(A(x_0,x))=B(x_0,x)\succeq 0$.
\end{proof}

The assumptions in parts (2) and (3) of
Proposition~\ref{prop:containment=maps} can not be omitted in general, as the
next examples show.

\begin{ex} (1) Consider the two linear pencils
\begin{align*}
  A(x) = \begin{bmatrix} -3+x_1+x_2 & 0 & 0 \\ 0 & -1+x_1 & 0 \\ 0 & 0 &
    -1+x_2 \end{bmatrix} 
  \text{ and }
  B(x) = \begin{bmatrix} -1+x_1+x_2 & 0 & 0 \\ 0 & x_1 & 0 \\ 0 & 0 & x_2
    \end{bmatrix} 
\end{align*}
defining unbounded, nonempty polyhedra in $\R^2$. It is easy to see that the
coefficient matrices are linearly independent and $S_A$ does not contain the
origin. 

While $S_A$ is contained in $S_B$, the linear map $\Phi_{AB}$ is not positive.
Indeed, the homogeneous pencil $ A(x_0,x) $ evaluated at the point 
$ (x_0,x_1,x_2) = (-1,-1/2,-1/2) $ is positive definite while $ B(x_0,x) $ is
indefinite. 

Therefore, the boundedness assumption in part (3) of
Proposition~\ref{prop:containment=maps} can not be omitted in general.
Using the extended linear pencil $ \widea(x) = 1 \oplus A(x) $ instead of
$A(x)$, the resulting constraint $ x_0 \geq 0 $ yields the positivity of
$ \widephi_{AB} $.
In fact, $ \widephi_{AB} $ is completely positive, which can be checked by the
SDFP~\eqref{eq:inclusion} as introduced in the next subsection.

\smallskip
\noindent (2) Consider the two linear pencils
\begin{align*}
  A(x) = \begin{bmatrix} x & 1 \\ 1 & 0 \end{bmatrix} 
  \text{ and }
  B(x) = \begin{bmatrix} 1 & -x \\ -x & 1 \end{bmatrix} .
\end{align*}
with linearly independent coefficient matrices. The corresponding spectrahedra
are the empty set, $ S_A = \emptyset $, and the interval $ S_B = [-1,1] $.
Thus $ S_A \subseteq S_B $. However, the linear map $\Phi_{AB}$ is not
positive, since the homogeneous pencil $ A(x_0,x) $ is positive semidefinite
at $ (x_0,x) = (0,1)$ but $ B(0,1) $ is not. Note that this holds for the
extended pencil as well. Thus nonemptyness of the inner spectrahedron can not
be dropped.
\end{ex}

\begin{rem} \label{lem:containment=pos_maps}
If our setting were changed from the case of linear subspaces to the case of
affine subspaces, with a natural adaption of the notion of positivity to
affine maps, Proposition~\ref{prop:containment=maps} had a slightly easier
formulation and proof:
Let $ A(x)\in\sym_k[x] $ and $ B(x)\in\sym_l[x] $. Define the affine subspaces 
$ \bar{\cal{A}} = \frac{1}{n} A_0 + \lin(A_1, \ldots, A_n) $ 
and $ \bar{\cal{B}} = \frac{1}{n} B_0 + \lin(B_1, \ldots, B_n) $ 
for linearly independent $ A_1,\ldots, A_n $. Then $ S_A \subseteq S_B $ if
and only if the affine function 
$ \bar{\Phi}_{AB}:\ \bar{\cal{A}}\rightarrow\bar{\cal{B}} $ 
defined by $ \frac{1}{n} A_0 + A_i\mapsto \frac{1}{n} B_0 + B_i $  for
$ i = 1,\ldots,n $ is positive.

\begin{proof} 
First, let $\bar{\Phi}_{AB}$ be positive and let $x \in S_A$.
Since $\bar{\Phi}_{AB}$ is positive, we have 
$ B(x) = \bar{\Phi}_{AB}(A(x)) \succeq 0 $, thus $x \in S_B$. 
Conversely, let 
$ \frac{1}{n} A_0 + \sum_{p=1}^n x_p A_p \in\bar{\cal{A}}\cap\psd_k $.
Then $nx\in S_A\subseteq S_B$ and hence 
$ \bar{\Phi}_{AB}(\frac{1}{n} A_0 + \sum_{p=1}^n x_p A_p) 
  = \frac{1}{n} B_0 + \sum_{p=1}^n x_p B_p \succeq 0 $.
\end{proof}
\end{rem}

\subsection{Connection between complete positivity and containment criteria}
\label{sec:contain1}

Choosing a basis of $\cal{A}$, we get a representation of the operator map
$\Phi_{AB}$ and by applying Proposition~\ref{prop:containment=maps}, we can use
the hierarchy defined in the last section to test positivity of $\Phi_{AB}$. 

To keep the notation simple, we assume boundedness and nonemptyness of $S_A$,
and only work with the map $\Phi_{AB}$. All statements can be given in the
general case using the map $\widehat\Phi_{AB}$. 
As seen before (see Section~\ref{sec:pos_maps}), every extension 
$ \tilde{\Phi}_{AB} $ of the linear map $\Phi_{AB}$ to the full matrix spaces
corresponds to a matrix 
$ C = (\tilde{\Phi}_{AB}(E_{ij}))_{i,j = 1}^k \in \sym_{kl} $
perceiving $C$ as a symmetric block matrix consisting of $k \times k$ blocks
$C_{ij}$ of size $l \times l$. 
Since $ A_0,\ldots,A_n $ and $ B_0,\ldots,B_n $ are generators of $\cal{A}$ and 
$\cal{B}$, respectively, some entries of $C$ are defined via 
$ B_p = \sum_{i,j=1}^k a_{ij}^p C_{ij} $ for $ p = 0,\ldots,n $. 

By Proposition~\ref{prop:containment=maps}, the polynomial optimization problem
from Proposition~\ref{prop:containment=psdp} can be translated to the problem
\begin{align}
\begin{split}
  \inf\ &\ z^T B(x) z \\
  \text{s.t.}\ & B(x) = \sum_{i,j=1}^k (A(x))_{ij} C_{ij} \\
  & G_A(x,z) \succeq 0 .
\end{split}
\label{eq:contain_pos}
\end{align}

Moreover, (an extension of) $\Phi_{AB}$ is completely positive if and only if
the matrix $ C = (\tilde\Phi_{AB}(E_{ij}))_{i,j = 1}^k \in \sym_{kl} $ is
positive semidefinite, i.e., if and only if the SDFP
\begin{equation}
  \label{eq:inclusion}
  C = \left( C_{ij} \right)_{i,j=1}^{k} \succeq 0
  \text{ and } 
  B_{p} = \sum_{i,j=1}^{k} a^{p}_{ij} C_{ij}
  \text{ for } p = 0,\ldots,n
\end{equation}
has a solution. So checking if there exists a positive semidefinite 
$ C \in\sym_{kl} $, gives another sufficient criterion for the containment
question. This is the method described in~\cite{Helton2010,Kellner2012}.

\begin{prop}
\label{prop:completepos1}
\cite[Theorem 4.3]{Kellner2012}
If the SDFP~\eqref{eq:inclusion} has a solution $C \succeq 0$, 
then $ S_A \subseteq S_B $.
\end{prop}

In terms of the linear pencils, the previous proposition states that the pencil 
$ B(x) = \sum_{ij=1}^{k} (A(x))_{ij} C_{ij} $ is positive semidefinite if 
$ A(x)$ and $C$ are positive semidefinite, rendering the objective polynomial
in~\eqref{eq:contain_pos} nonnegative on the set $S_A \times \T_{r,R}(0)$.

As we will see next, positive semidefiniteness of the matrix $C$ is not only a
sufficient condition for containment and thus for the nonnegativity of the
polynomial optimization problem in Proposition~\ref{prop:containment=psdp}, but
also for its relaxations~\eqref{eq:contain_lasserre}
and~\eqref{eq:contain_sos}. 

We show the following result:

\begin{thm} \label{thm:implications1}
Let $ S_A \neq\emptyset $. Then for the properties 
\begin{enumerate}
  \item[(1')] 
    $\widephi_{AB}$ is completely positive,
  \item[(1)] 
    the SDFP~\eqref{eq:inclusion} has a solution $C \succeq 0$,
  \item[(2')]
  $ \lambda_{\sos}(0) \geq 0 $ 
  (and thus $\lambda_{\sos}(t) \ge 0$ for all $t \ge 0$),
  \item[(2)] 
  $\mu_{\mom}(2) \ge 0$ (and thus $\mu_{\mom}(t) \ge 0$ for all $t \ge 2$),
  \item[(3)] 
  $S_A \subseteq S_B$,
  \item[(3')]
  $\widephi_{AB}$ is positive,
\end{enumerate}
we have the implications and equivalences
\begin{align*}
  (1') \Longleftarrow (1) \iff (2') \Longrightarrow (2) \Longrightarrow  
  (3) \iff (3') 
\end{align*}
with the first implication an equivalence whenever $ \widehat{\cal{A}}$
contains a positive definite matrix.

If, in addition, $S_A$ is bounded, then $\widephi_{AB}$ in (1') and (3')
can be replaced by $\Phi_{AB}$.
\end{thm}

Note that if the spectrahedron $S_A$ is bounded, then
Theorem~\ref{thm:convergence} implies a partial converse of the implication 
$ (2)\Longrightarrow (3) $. 
Namely, if $ \emptyset \neq S_A \subseteq S_B $ and $ S_A $ is bounded, then 
$ \mu_{\mom}(t) \uparrow \mu \geq 0 $ for $ t\rightarrow \infty $. 

Recalling Corollary~\ref{co:inclusion1} and Proposition~\ref{prop:completepos1},
the remaining task is to prove $ (1) \Longrightarrow (2) $ and 
$ (1) \iff (2') $. 
The first is achieved in the following theorem. The proof of the second
statement is straightforward. Indeed, by an easy computation one can check
that for $t=0$ the sos-matrix $S(x)$ is equal to (a permutation of) the matrix
$C$ coming from the SDFP \eqref{eq:inclusion} applied to the extended pencil.

\begin{thm}
\label{thm:connection1}
If the SDFP~\eqref{eq:inclusion} has a solution, then the infimum
$\mu_{\mom}(2)$ of the initial relaxation in~\eqref{eq:contain_lasserre} is
nonnegative.
\end{thm}

\begin{proof}
Assume $C\succeq 0$ is a solution to the SDFP. 
Define the matrix $C'$ via $ (C'_{st})_{i,j} = (C_{ij})_{s,t} $, i.e., a
$ kl\times kl $ block matrix consisting of $l\times l$ blocks of size
$ k\times k $. Since it arises by permuting rows and columns of $C$
simultaneously, $C'$ is positive semidefinite as well.

Since~\eqref{eq:contain_poly} is feasible, the SDP~\eqref{eq:contain_lasserre}
is feasible as well. For any $(x,z)$, the linearity of $L_{y}$ implies for the
objective in~\eqref{eq:contain_lasserre} (see also~\eqref{eq:contain_pos})
\begin{align*}
\begin{split}
  L_{y}(z^T B(x) z) 
  & =  L_{y} \left( z^T \sum_{i,j=1}^k (A(x))_{ij}C_{ij} z \right) 
   =  L_{y} \left( \sum_{i,j=1}^k \sum_{s,t=1}^l z_s z_t  (A(x))_{ij}
(C_{ij})_{s,t} \right) \\
  & =  \sum_{i,j=1}^k \sum_{s,t=1}^l L_{y} \left( z_s z_t  (A(x))_{ij}
\right) (C_{ij})_{s,t} 
   = \mathds{1}^T \left( L_{y} \left( zz^T \otimes A(x) \right) \odot C'
\right) \mathds{1} ,
\end{split}
\end{align*}
where $\odot$ denotes the Hadamard product and $\mathds{1}\in\R^{kl}$ is the
all-one vector.
 
In the Hadamard product, the first matrix is positive semidefinite as a
principal submatrix of
$ 
  M_{1}( G_A y ) 
  = L_{y} \left( b_1(x,z) b_1(x,z)^T \otimes \diag(A(x),g_r(z),g^R(z)) \right) 
$
and $C' \succeq 0$ as stated above.
By the Schur product theorem (see \cite[Theorem 7.5.3]{Horn1994}), the Hadamard
product of the two matrices is positive semidefinite as well. Hence, 
$ L_{y}(z^T B(x) z) \ge 0 $ for any feasible $y$, and $ \mu_{\mom}(2) \ge 0 $.
\end{proof}

\begin{rem} \label{rem:connection1}
a) Theorem~\ref{thm:connection1} can be stated for the polynomial
optimization problem~\eqref{eq:contain_poly} itself. The proof is the same
without the linearization operator $L_{y}$.

\noindent b) The reverse implication in Theorem~\ref{thm:connection1} (and
Theorem~\ref{thm:implications1}), i.e., $ (2) \Longrightarrow (1)$, is not
always true. Example~\ref{ex:ballelliptope} serves as a counterexample.

\noindent c) In terms of positive linear maps (see
Section~\ref{sec:pos_maps}), Theorem~\ref{thm:connection1} states that
$k$-positivity is a sufficient condition for the initial relaxation step to
certify containment. More generally, one can ask about the exact relationship
between the exactness of the $t$-th relaxation step and $(k+2-t)$-positivity
of $\Phi_{AB}$. 
\end{rem}

As seen in the proof of the last theorem, we can always represent the objective
function of the optimization problem~\eqref{eq:contain_lasserre} in terms of a
submatrix of $M_{1}( A y )$ and the matrix $C'$ (where the last one arises by
permuting rows and columns of $C$ simultaneously),
$
  L_{y}(z^T B(x) z) 
  = \mathds{1}^T \left( L_{y} \left( zz^T \otimes A(x) \right) \odot C'
  \right) \mathds{1}.
$
In fact, this expression is just the trace or, equivalently, the scalar product
of these two matrices, i.e.,\ 
\[
  L_{y}(z^T B(x) z) 
  = \tr \left( L_{y} \left( zz^T \otimes A(x) \right) \cdot C' \right)
  = \left\langle  L_{y} \left( zz^T \otimes A(x) \right) , C'
  \right\rangle .
\]
Since $ L_{y} \left( zz^T \otimes A(x) \right) $ is a principal submatrix of
$M_{1}( G_A y )$ which is constrained to be positive semidefinite, the first
entry in the scalar product is positive semidefinite. Therefore, the question 
of whether the objective function is nonnegative on the feasible region reduces
to the question of which conditions on the matrix $C$ (or $C'$) guarantee the
nonnegativity of the scalar product on this set.

Using Theorem~\ref{thm:connection1}, we can extend the exactness results 
from~\cite{Kellner2012} to the hierarchy~\eqref{eq:contain_lasserre}, i.e., in
some cases already the initial relaxation is not only a sufficient condition but
also necessary for containment. More precisely, in these cases the equivalences 
$ (1) \iff (2') \iff (2) \iff (3) $ hold in Theorem~\ref{thm:implications1}.
These results rely on the specific pencil representation of the given
spectrahedra. Before stating the results, we have to agree on a consistent
representation. 

Every polyhedron $P = \{x \in \R^n \, : \, b + Ax \ge 0 \}$ has a natural
representation as a spectrahedron:
\begin{equation}
  P = P_A = \left\{ x \in \R^n \ : \  A(x) = 
  \begin{bmatrix} 
    a_1(x) & 0 & 0 \\ 0& \ddots & 0 \\ 0 & 0 & a_k(x)
  \end{bmatrix}
  \succeq 0 \right\},
\label{eq:polytope}
\end{equation}
where $a_i(x)$ abbreviates the $i$-th entry of the vector $b+Ax$. 
$P_A$ contains the origin if and only if the inequalities can be scaled so that
$b = \mathds{1}_k$, where $\mathds{1}_k$ denotes the all-ones vector in $\R^k$.
Hence, in this case, $A(x)$ is monic, and it is called the \emph{normal form} 
of the polyhedron $P_A$.

A centrally symmetric ellipsoid with axis-aligned semiaxes of lengths 
$ a_1,\ldots, a_n $ can be written as the spectrahedron $S_A$ of the monic
linear pencil
\begin{equation}
  \label{eq:ellipsoid}
  A(x) \ = \ I_{n+1}
       + \sum_{p=1}^{n} \frac{x_p}{a_p}(E_{p,n+1}+E_{n+1,p}).
\end{equation}
We call~\eqref{eq:ellipsoid} the \emph{normal form of the ellipsoid}.
Specifically, for the case of all semiaxes having the same length
$ \nu := a_1 = \cdots = a_n $, this gives the \emph{normal form of a ball} with
radius $\nu$.

We are now ready to state the exactness results in
Corollaries~\ref{co:exactness} and~\ref{co:scaling}.

\begin{cor}
\label{co:exactness}
Let $ A(x)\in\sym_k[x] $ and $ B(x)\in\sym_l[x] $ be linear pencils. 
In the following cases, the initial relaxation step $(t=2)$
in~\eqref{eq:contain_lasserre} certifies containment of $S_A$ in $S_B$.
\begin{enumerate}
  \item 
  if $A(x)$ and $B(x)$ are normal forms of ellipsoids
  (both centrally symmetric, axis-aligned semiaxes),
  \item 
  if $A(x)$ and $B(x)$ are normal forms of a ball 
  and an $\HH$-polyhedron, respectively, 
  \item 
  if $B(x)$ is the normal form of a polytope,
  \item 
  if $\widea(x)$ (see \eqref{eq:extended_pencil}) is the extended form of 
  a spectrahedron and $B(x)$ is the normal form of a polyhedron.
\end{enumerate}
\end{cor}

\begin{proof}
Follows directly from~\cite[Theorem 4.8]{Kellner2012} and
Theorem~\ref{thm:connection1}.
\end{proof}

Statements (2) to (4) in the previous corollary can also be deduced from the
point of view of positive and completely positive maps introduced in
Section~\ref{sec:pos_maps}. This follows from a result 
in~\cite[Proposition 1.2.2]{Arveson1969} stating that positive maps into
commutative \ca s are completely positive. The second exactness result states
that the initial relaxation step can always certify containment of a scaled
situation.

\begin{cor}
\label{co:scaling}
Let $ A(x)\in\sym_k[x] $ and $ B(x)\in\sym_l[x] $ be monic linear pencils such
that $S_A$ is bounded. 
Then there exists $\nu > 0$ such that the initial relaxation step certifies 
$ \nu S_A \subseteq S_B $, where 
$ \nu S_A = \{x \in \R^n \, : \, A^{\nu}(x) := A(\frac{x}{\nu}) \succeq 0\} $ 
is the scaled spectrahedron.
\end{cor}

\begin{proof} 
This follows from \cite[Proposition~6.2]{Kellner2012} and
Theorem~\ref{thm:connection1}.
\end{proof}

\section{Numerical experiments}
\label{sec:application}

While the complexity of the containment question for spectrahedra is co-NP-hard
in general, the relaxation techniques introduced in this paper give a practical
way of certifying containment. We implemented the hierarchy and applied it to
several examples. The criterion performs well already for relaxation orders as
low
as $t=2,3$, as we will witness throughout this section.

We start by reviewing an example from~\cite{Kellner2012} in 
Section~\ref{sec:diskcontain}, showing that the new hierarchical relaxation 
indeed outperforms the complete positivity relaxation. We then give an 
overview on the performance of the relaxation on some more examples.

In Section~\ref{sec:random}, we compare the results from the moment approach
\eqref{eq:contain_lasserre} to the alternative sum-of-squares approach
\eqref{eq:contain_sos}. To assess the performance of both algorithms, we
compare results as well as runtimes of the algorithms on randomly generated
pencils of varying sizes. 

We use the example of computing the symmetric
circumradius of a spectrahedron to show how the relaxation can be simplified in
the case when the outer spectrahedron can be described as the positivity region
of a single polynomial. This is discussed in Section~\ref{sec:radii}.

For our computations, we modeled the hierarchy using high-level
YALMIP~\cite{YALMIP, Loefberg2009} code. We used
MOSEK 7 \cite{Andersen2003} as external solver for the optimization problems
defined in YALMIP. The \textsc{Matlab} version
used was R2011b, running on a desktop computer with Intel Core i3-2100 @ 3.10
GHz and 4 GB of RAM. 

Throughout this section, we use the following notation. As before, integer $n$
stands for the number of variables in the pencils, $k$ and $l$ for the size of
the pencil $A(x)$ and $B(x),$ respectively. For monic pencils, we examine
$\nu-$scaled spectrahedra $\nu S_A$ as defined in Corollary~\ref{co:scaling}. 
We denote the (numerical) optimal value of the moment relaxation
\eqref{eq:contain_lasserre} of order $t$ by $\mu_{\mom}(t)$, the (numerical)
optimal value of the alternative relaxation \eqref{eq:contain_sos} of of order
$t$ by $\lambda_{\sos}(t)$ . 
In the tables, ``sec`` states the time in seconds for setting up the problem in
YALMIP and solving it in MOSEK.

If not stated otherwise, the inner radius is set to $ r=1 $ and outer radius
to $R=2$ in relaxation~\eqref{eq:contain_lasserre}.

\subsection{Numerical computations}
\label{sec:diskcontain}
We review the example of containment of two disks from \cite{Kellner2012}. The
complete positivity criterion from that work certifies the containment only if
the disk on the inside is scaled small enough. Theorem~\ref{thm:connection1}
shows that any containment certified by the complete positivity criterion is
certified by the hierarchical relaxation. In the following example we go one
step further, showing that the latter performs strictly better than the
feasibility criterion already in small relaxation orders.

 \begin{ex}\label{ex:diskdisk}
Consider the monic linear pencils 
$ 
  A^\nu (x) = I_3 + x_1 \frac{1}{\nu} ( E_{1,3} + E_{3,1} ) 
    + x_2 \frac{1}{\nu} (E_{2,3} + E_{3,2} ) 
  \in \sym_3 [x] 
$ 
with parameter $ \nu > 0$ and 
$
  B(x) = I_2 + x_1 ( E_{1,1} - E_{2,2} ) + x_2 ( E_{1,2} + E_{2,1} )
 \in \sym_2[x] .
$
The spectrahedra defined by the pencils are the disk of radius $\nu > 0$
centered
at the origin, $ \nu S_A = \B_{\nu}(0)$, and the unit disk $S_B = \B_1(0)$,
respectively. Clearly, $ \nu S_A \subseteq S_B$ if and only if $ 0 < \nu
\leq 1$. In
particular, for $ \nu = 1 $, both pencils define the unit disk $\B_1(0) = S_A =
S_B$.

In~\cite[Section 6.1]{Kellner2012} it is shown that the complete positivity
criterion for the containment problem $ \nu S_A \subseteq S_B$ is satisfied if 
$ 0 < \nu \leq \frac{1}{2} \sqrt{2}$.
Remarkably, the performance of relaxation~\eqref{eq:contain_lasserre} depends
on the choice of the parameters $r$ and $R$. Table~\ref{tab:numeric-disk}
contrasts the results of the moment relaxation with parameters $r = 1, R = 2$
with the results of the complete positivity criterion for the problem $\nu S_A
\subseteq S_B$. 
Our numerical computations shows that the semidefinite relaxation of order $t=2$
certifies the same cases as the complete positivity criterion. For $t =3$ we
have exactness of the criterion. 

\begin{table}
\begin{tabular}{l|c|rr|rr}
  \toprule
  $\nu$ & SDFP~\eqref{eq:inclusion} & $\mu_{\mom}(2)$ & sec & $\mu_{\mom}(3)$
& sec \\
  \midrule
0.7&feasible&0.0101&0.11&0.300&0.17\\
0.707&feasible&0.000151&0.1&0.293&0.16\\
$1/\sqrt{2}$&feasible&7.29$\cdot 10^{-11}$&0.09&0.293&0.16\\
0.708&infeasible&-0.000632&0.1&0.292&0.16\\
0.8&infeasible&-0.0657&0.09&0.200&0.16\\
1&infeasible&-0.207&0.1&9.78$\cdot 10^{-09}$&0.19\\
1.1&infeasible&-0.278&0.1&-0.100&0.16\\
  \bottomrule
\end{tabular}
\\[+0.5ex]
\caption{Disk $\nu S_A$ in disk $S_B$ for two different representations and
various radii $\nu$ of the inner disk as described in Example~\ref{ex:diskdisk}.
}
\label{tab:numeric-disk}
\end{table}

When choosing $r = R = 1$, the semidefinite
relaxation~\eqref{eq:contain_lasserre} is exact already for relaxation order
$t=2$ and returns the same optimal values as for relaxation order $t=3$. 
This choice of parameters however leads to numerical problems in the solver
occasionally. Furthermore the example of the two disks is the only one we have
found, where results for orders $t=2$ and $t=3$ differ if $r$ and $R$ are
chosen distinct. In all other examples, results seem to be exact already for
$t=2$. Therefore we advise to use $r = 1$ and $R = 2$ in general applications.
\end{ex}

In the next example, we examine the containment of a ball in an elliptope. The
elliptope is a nice example of a spectrahedron that is described by a pencil
consisting of very sparse matrices. While the pencil is of small size, it
is occupied by a large number of variables. 

\begin{ex} \label{ex:ballelliptope}
For this example, the pencil description of the ball is as
in~\eqref{eq:ellipsoid}. The \emph{elliptope}~\eqref{eq:elliptope} can be
described as the positivity domain of a symmetric pencil with ones on the
diagonal and distinct variables in the remaining positions; see~\cite[Section
2.1.3]{bpt-2013}. 

As exhibited in Table \ref{tab:numeric-containment}, the ball of radius
$\frac{1}{2}$ in dimensions $n = 3, 6, 10$ and $15$ is contained in the
elliptope of the respective dimension. The computational time grows in the
number of variables, but even dimensions as high as 15 are in the scope of
desktop computers if the size $l$ of the pencil $B(x)$ is moderate. 

\begin{table}
\begin{tabular}{rrr|rr|rr}
  \toprule
 \multicolumn{3}{c}{size}& \multicolumn{2}{|c}{objective value}
&\multicolumn{2}{|c}{sec} \\
$n$ & $k$ & $l$ & $\mu_{\mom}(2)$ & $\lambda_{\sos}(0)$ &
$\mu_{\mom}(2)$ & $\lambda_{\sos}(0)$ \\
 \midrule
 3 &  4 & 3 & 0.293 & 0.293 &   0.12 & 0.81\\
 6 &  7 & 4 & 0.134 & 0.134 &   1.46 & 0.77\\
10 & 11 & 5 & 0.106 & -$3.972\cdot 10^{-8}$ & 28.51 & 3.61\\
15 & 16 & 6 & 0.087 & -0.118 & 588.57 & 65.47 \\
  \bottomrule
\end{tabular}
\\[+0.5ex]
\caption{Computational test of containment of ball in elliptope as described in
Example~\ref{ex:ballelliptope}.}
\label{tab:numeric-containment}
\end{table}

Interestingly, while the moment relaxation is slower than the sum-of-squares
approach in this example (for $t=0$ and $t=2$, respectively), the latter
approach fails to be exact in dimension $n=10,15$. When trying to compute the
next relaxation step $\lambda_{\sos}(1)$ for $(n,k,l)=(10,11,5)$, we stopped
the computation after about 15 hours. 
Note that the SDFP~\eqref{eq:inclusion} is solvable for $(n,k,l)=(10,11,5)$ but
not solvable for $(n,k,l)=(15,16,6)$. Thus, for $(n,k,l)=(15,16,6)$, this
example serves as a counterexample for the reverse statement of
Theorem~\ref{thm:connection1} (or, equivalently, for the implication $ (2)
\Longrightarrow (1) $ in Theorem~\ref{thm:implications1}).
\end{ex}

\begin{ex}
Consider the linear map
\[
  \Phi:\ \sym_{3}\to\sym_{3},\ A\mapsto 2\begin{bmatrix}
  A_{11}+A_{22} & & \\ & A_{22}+A_{33} & \\ & & A_{33}+A_{11} \end{bmatrix} -A .
\]
Due to Choi~\cite{Choi1975bi}, the map $\Phi$ is (1- and 2-)positive but not
completely positive. Indeed, the SDFP~\eqref{eq:inclusion} is not feasible.
Using hierarchy~\eqref{eq:contain_lasserre} (with $r=R=1$), the initial
relaxation step is also not feasible but for $t=3$ the relaxation yields a small
positive value implying positivity of $\Phi$.
\end{ex}

\subsection{Randomly Generated Spectrahedra} 
\label{sec:random}

We applied both hierarchical criteria, the moment hierarchy
\eqref{eq:contain_lasserre} and the alternative sum-of-squares approach
\eqref{eq:contain_sos} to several instances of linear pencils with random
entries. 

For the experiments in this section, we generate coefficient matrices $A_1,
\ldots, A_n$ by assigning random numbers
to the off-diagonal entries of the matrices. Numbers are drawn
from a uniform distribution on $[-1,1]$. 
The generated matrices are sparse in the sense that roughly $35\%$ of the
off-diagonal entries are nonzero. The matrix for the constant term, $A_0,$ is
generated in the same way, but features ones on the diagonal. This choice leads
to bounded spectrahedra in most cases, namely when the matrices $A_0, \ldots,
A_k$ are linearly independent. Unbounded spectrahedra and spectrahedra
without interior are discarded.

The pencil of the second spectrahedron $S_B$ is generated in the same way,
except that the diagonal entries of $B_0$ are chosen larger. This has the effect
that the corresponding spectrahedra are scaled and the containment 
$S_A \subseteq S_B$ is more likely to happen. 

\begin{ex} \label{ex:random}
We apply the hierarchies to a range of problems with varying dimensions
and pencil sizes as reported in Table \ref{tab:random}. 
\begin{table}
\small
 \begin{tabular}{l | r r r |rrrr|rrrr}
  \toprule
& \multicolumn{3}{c}{size}& \multicolumn{4}{|c}{objective value}
&\multicolumn{4}{|c}{sec} \\
no. & $n$ & $k$ & $l$ & $\mu_{\mom}(2)$ & $\mu_{\mom}(3)$ & 
$\lambda_{\sos}(0)$ & $\lambda_{\sos}(1)$ & $\mu_{\mom}(2)$&  $\mu_{\mom}(3)$ &
$ \lambda_{\sos}$(0) & $\lambda_{\sos}(1)$\\
 \midrule
1 & 2 & 4 & 4 & 0.330 & 0.330 & 0.330 & 0.330 & 0.26 & 2.49 & 0.29 & 2.04\\
2 & 2 & 6 & 4 & 1.459 & 1.459 & 1.459 & 1.459 & 0.16 & 2.95 & 0.34 & 9.98\\
3 & 2 & 4 & 6 & -2.009 & -2.009 & -2.009 & -2.009 & 0.38 & 31.03 & 0.42 &
10.61\\
4 & 2 & 6 & 6 & -0.209 & -0.209 & -0.209 & -0.209 & 0.36 & 31.53 & 0.72 &
76.23\\
5 & 3 & 4 & 4 & 0.156 & 0.156 & 0.156 & 0.156 & 0.20 & 6.35 & 0.30 & 3.83\\
6 & 3 & 6 & 4 & 0.332 & 0.332 & 0.332 & 0.332 & 0.22 & 8.86 & 0.34 & 24.52\\
7 & 3 & 4 & 6 & -6.918 & -6.906 & -6.918 & -6.918 & 0.82 & 117.3 & 0.45 & 28.3\\
8 & 3 & 6 & 6 & 0.028 & 0.028 & 0.028 & 0.028 & 0.66 & 84.71 & 0.71 & 207.84\\
9 & 4 & 4 & 4 & -3.164 & -3.164 & -3.164 & -3.164 & 1.33 & 32.64 & 0.97 &
10.19\\
10 & 4 & 6 & 4 & 0.593 & 0.593 & 0.593 & 0.593 & 0.32 & 27.88 & 0.35 & 66.39\\
11 & 4 & 4 & 6 & -0.938 & -0.938 & -0.938 & -0.938 & 1.21 & 326 & 0.45 & 64.41\\
12 & 4 & 6 & 6 & -0.251 & -0.251 & -0.251 & -0.251 & 1.43 & 317.08 & 0.81 &
567.07\\
  \bottomrule
 \end{tabular}
\\[+0.5ex]
\caption{Computational test of containment of randomly generated spectrahedra
as described in Example~\ref{ex:random}.  }
\label{tab:random}
\end{table}
To illustrate the approach, we provide the pencils for experiment no. 1 below.
\[ 
A(x) = \begin{bmatrix} 
1 & 0.2528x_1+0.3441x_2 & 0 & 0 \\
0.2528x_1+0.3441x_2 & 1 & 0 & -0.1314x_1 \\ 
0 & 0 & 1 & 0.7969x_2 \\  
0 & -0.1314x_1 & 0.7969x_2 & 1 \\ 
 \end{bmatrix}
\]

\[
B(x) = \begin{bmatrix}
2 & 0.8454 & 0 & 0 \\  
0.8454 & 2 & -0.2489x_1-0.4063x_2 & 0 \\ 
0 & -0.2489x_1-0.4063x_2 & 2 & 0.3562x_1 \\  
0 & 0 & 0.3562x_1 & 2 \\ 
\end{bmatrix}
\]

For this experiment with randomly generated matrices, the
truth value of the containment question is unknown a priori. In the case of a
positive
objective value, our criterion yields a certificate for the containment. For
negative objective values, we inspected plots of the spectrahedra
to check appropriateness of the criterion. Plots of the spectrahedra from the
two-dimensional experiments no.  1--4 are shown in Figure~\ref{fig:2dspecs}. 

 In cases of higher dimension ($n > 3$), we examined projections of the
spectrahedra. See Figure~\ref{fig:projections} for projections of the
spectrahedra from experiment no. 12 to different planes. The small negative
objective value reported in Table~\ref{tab:random} suggests that there is only a
small overlap of $S_A$ over the boundary of $S_B$. Indeed, the projections to
the coordinate planes suggest that $S_A$ is contained in $S_B$. But when
projecting to the plane spanned by $0.3 x_1 +x_2$ and $x_3$, we see that the
spectrahedra are not contained. 

\begin{figure}
\centering
\begin{subfigure}{.25\textwidth}
  \centering
  \includegraphics[width=\linewidth]{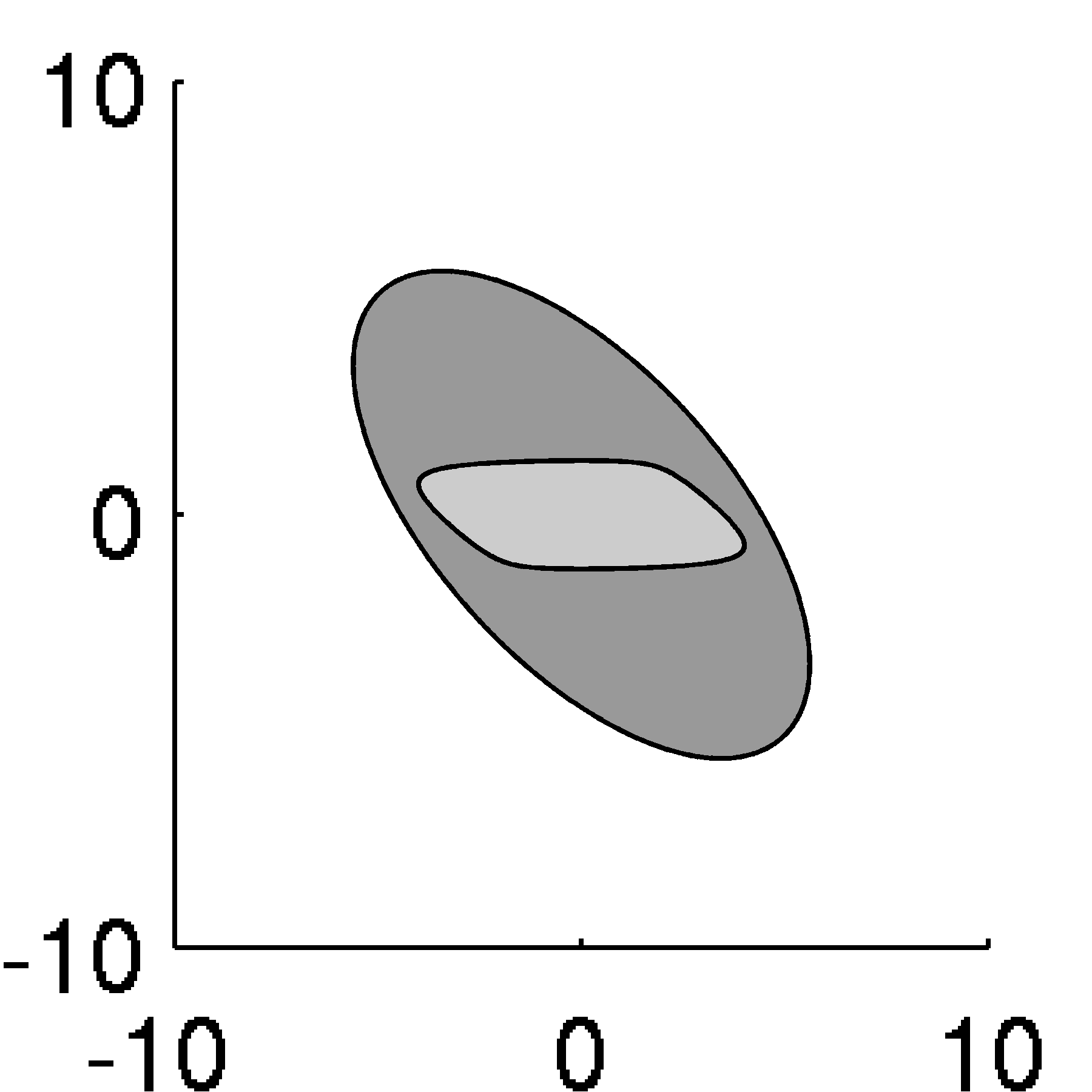}
\end{subfigure}%
\begin{subfigure}{.25\textwidth}
  \centering
  \includegraphics[width=\linewidth]{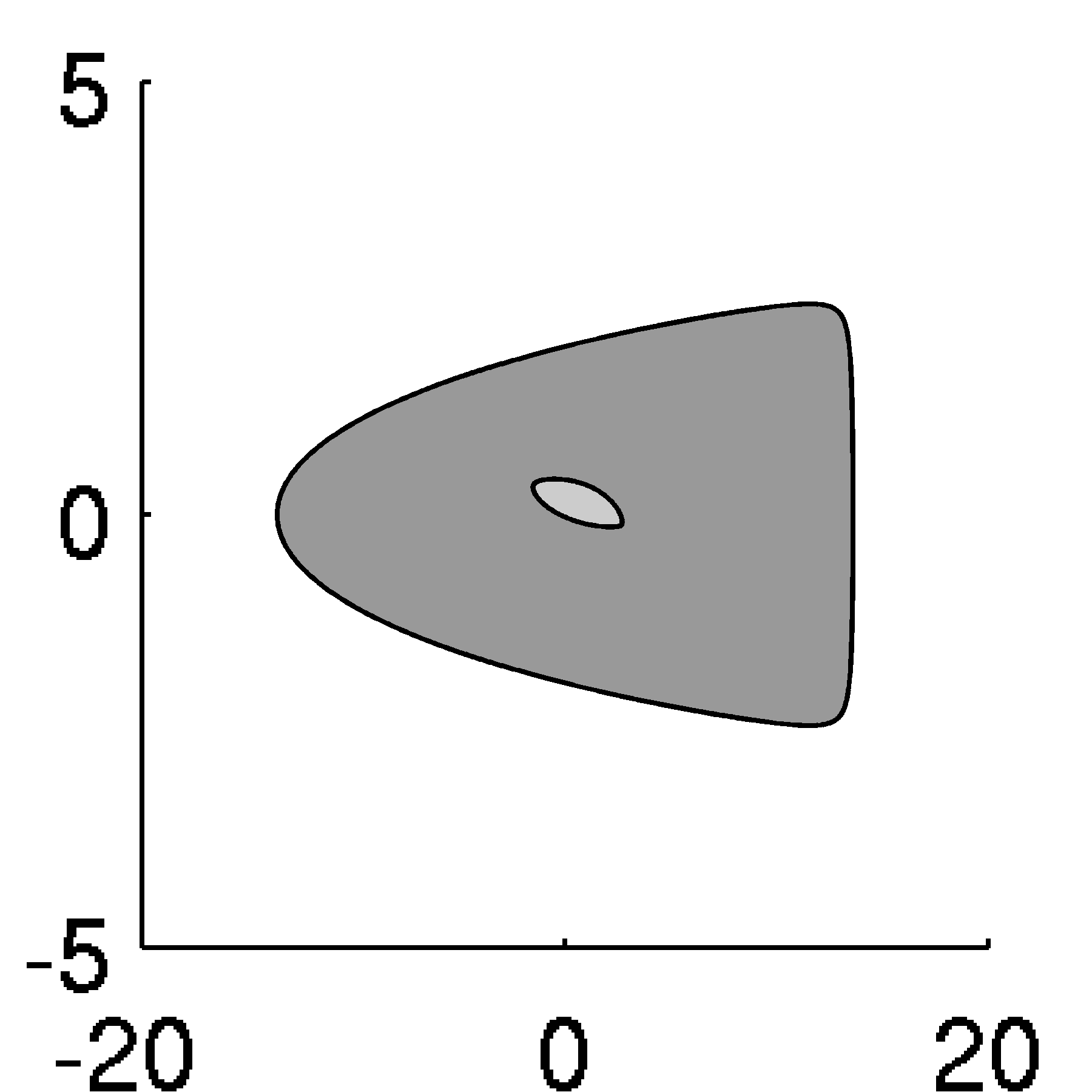}
\end{subfigure}%
\begin{subfigure}{.25\textwidth}
  \centering
  \includegraphics[width=\linewidth]{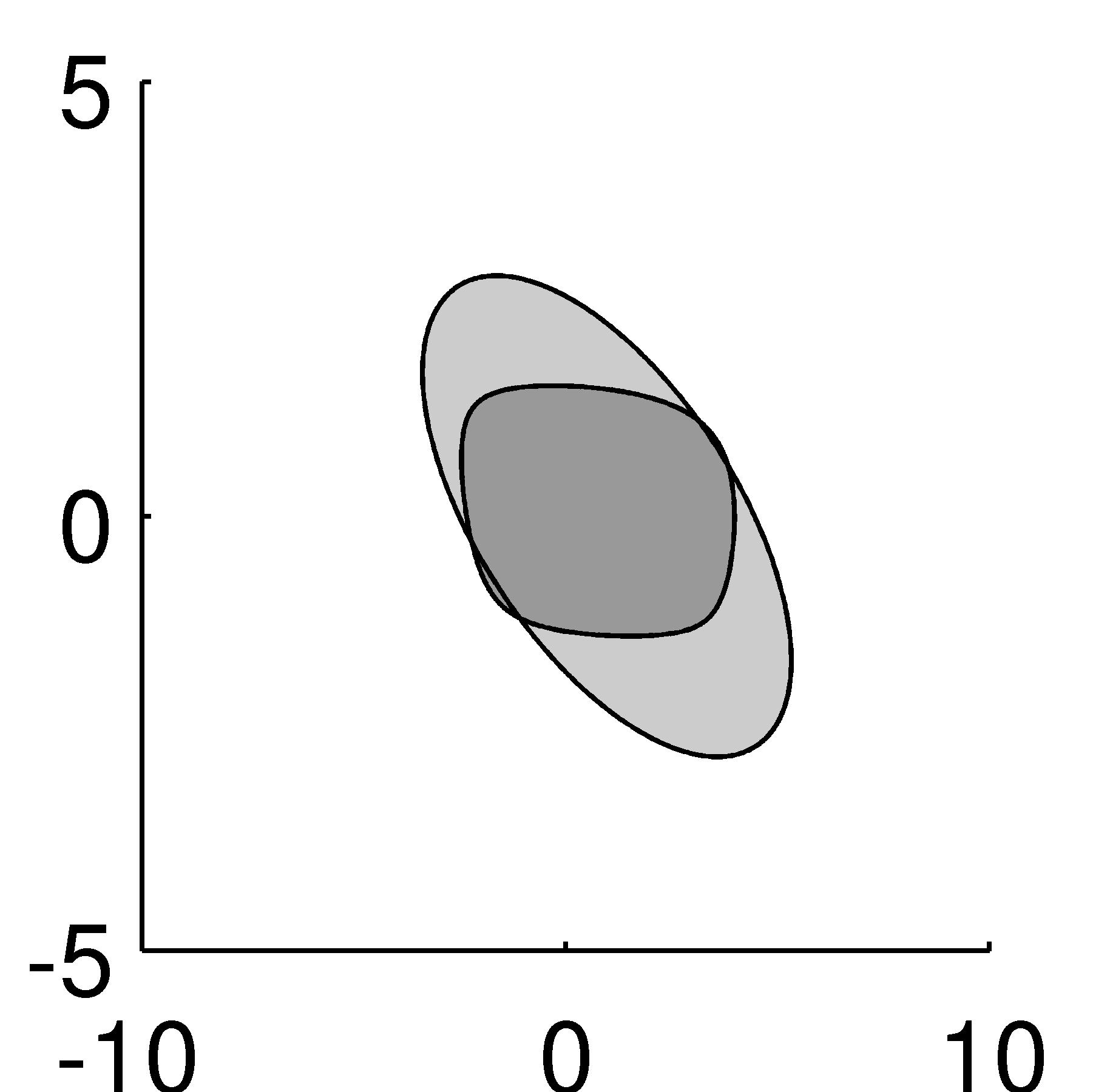}
\end{subfigure}%
\begin{subfigure}{.25\textwidth}
  \centering
  \includegraphics[width=\linewidth]{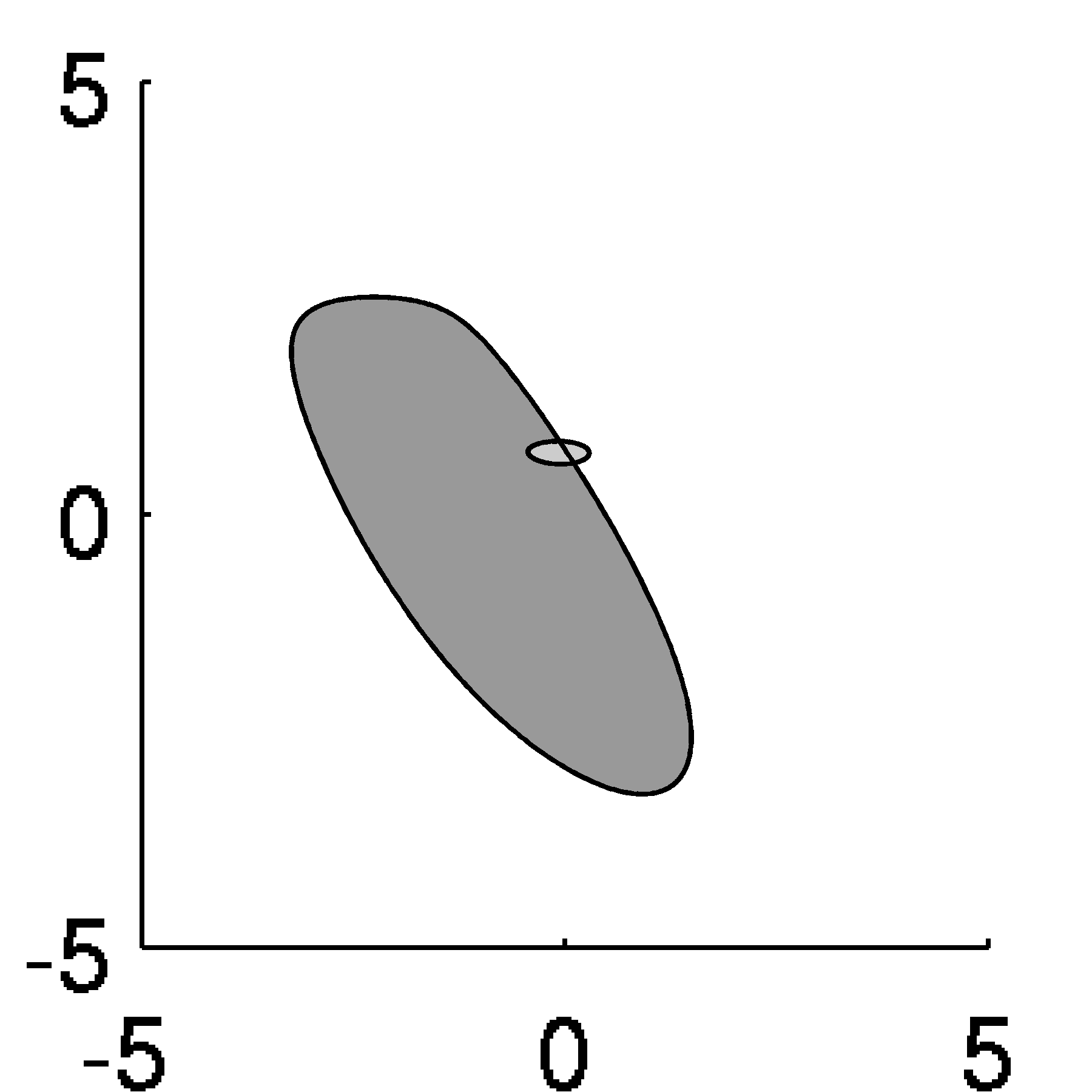}
\end{subfigure}%
\caption{Spectrahedra of experiments no. 1--4 from Table~\ref{tab:random}. 
$S_A$: light grey, $S_B$: dark grey. }
\label{fig:2dspecs}
\end{figure}

\begin{figure}
\centering
\begin{subfigure}{.25\textwidth}
  \centering
  \includegraphics[width=\linewidth]{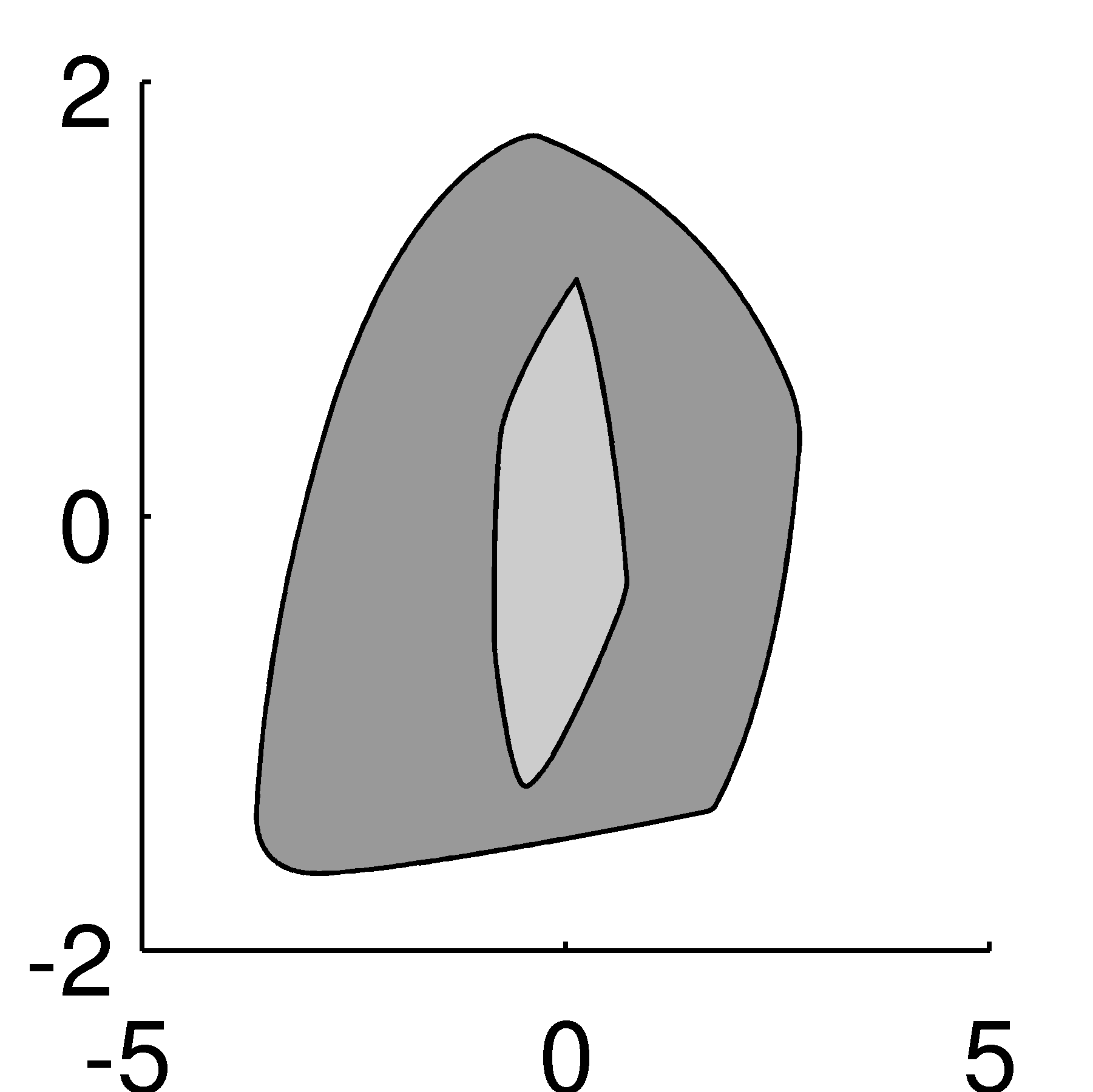}
\end{subfigure}%
\begin{subfigure}{.25\textwidth}
  \centering
  \includegraphics[width=\linewidth]{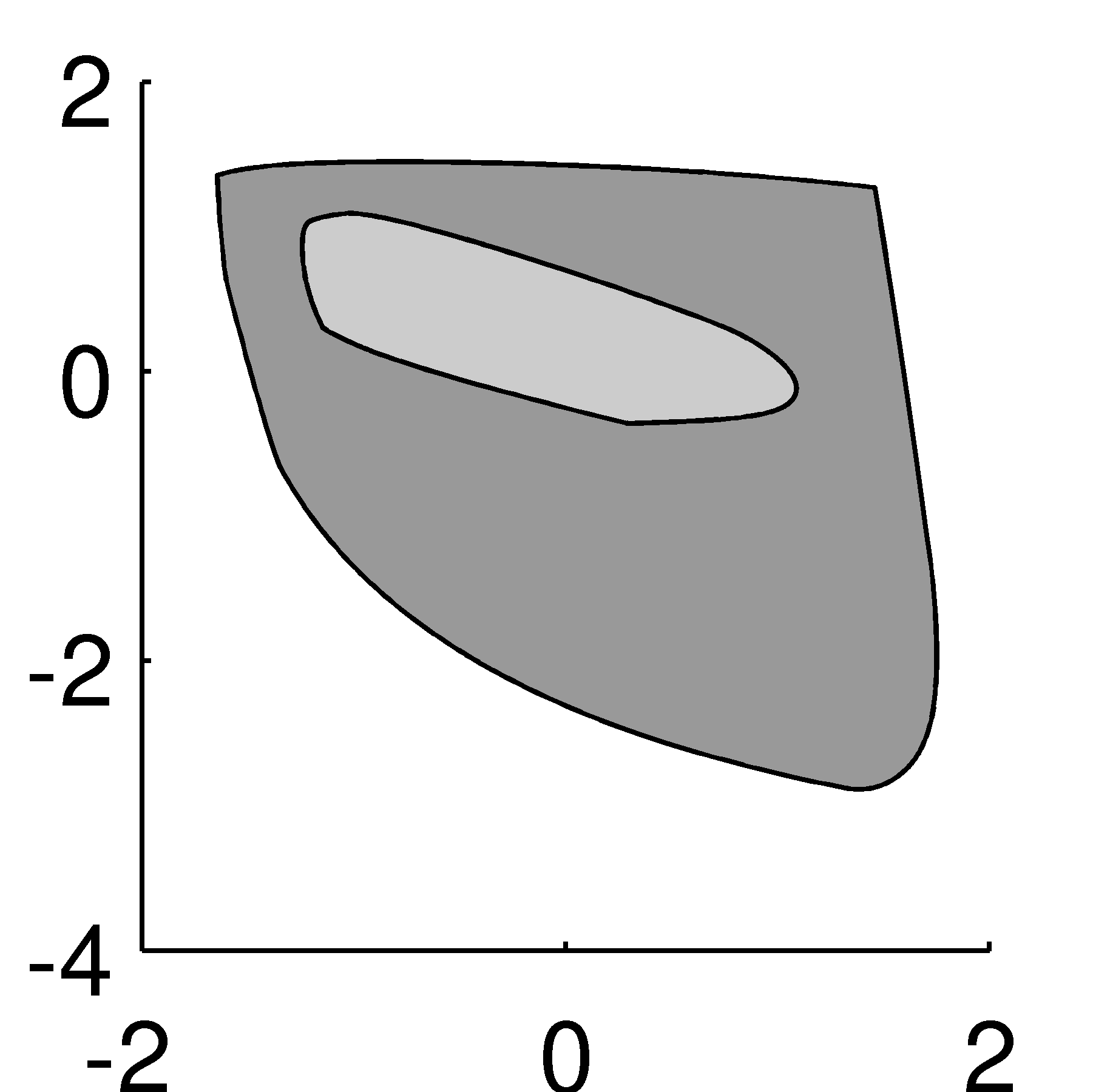}
\end{subfigure}%
\begin{subfigure}{.25\textwidth}
  \centering
  \includegraphics[width=\linewidth]{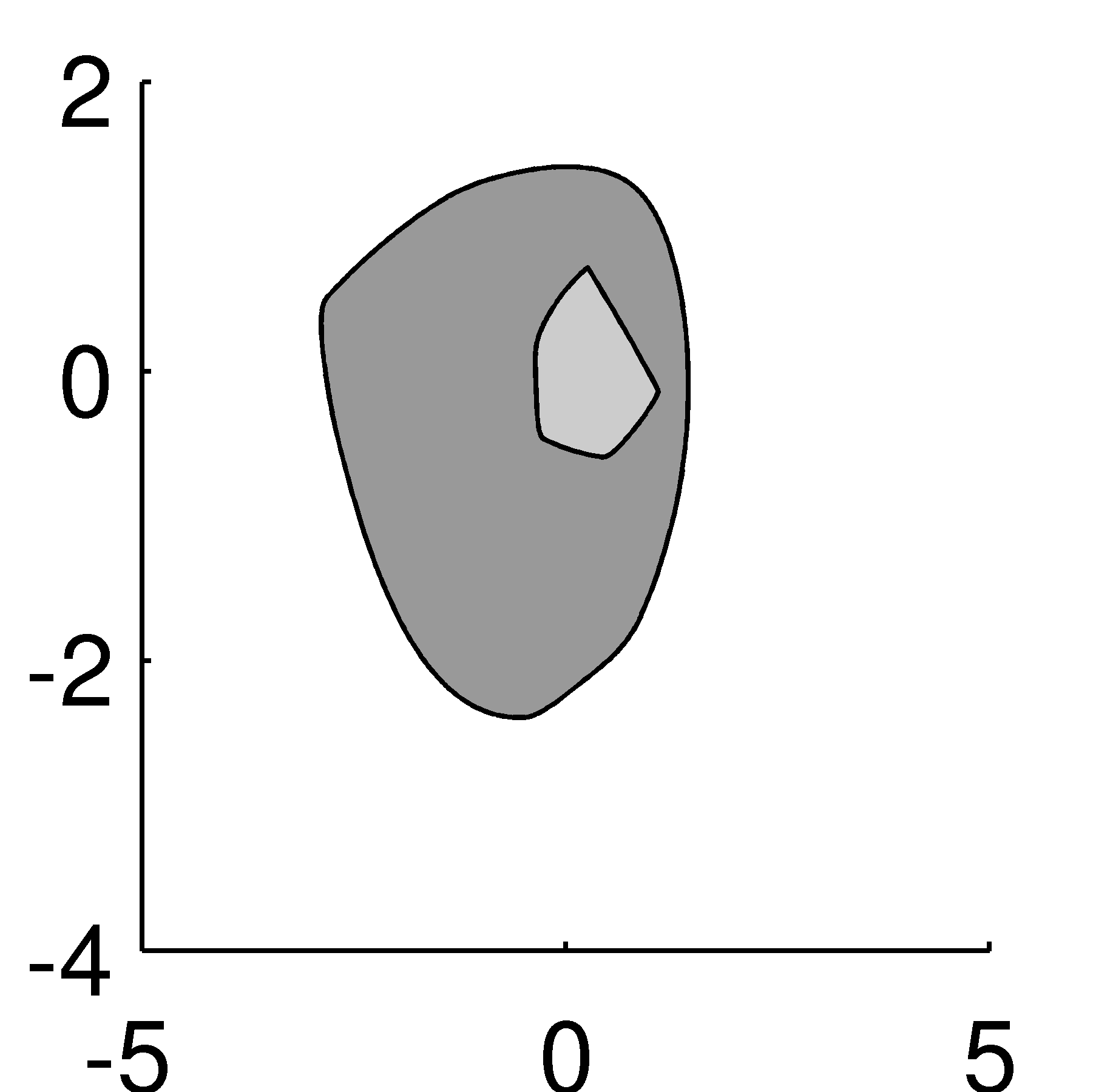}
\end{subfigure}%
\begin{subfigure}{.25\textwidth}
  \centering
  \includegraphics[width=\linewidth]{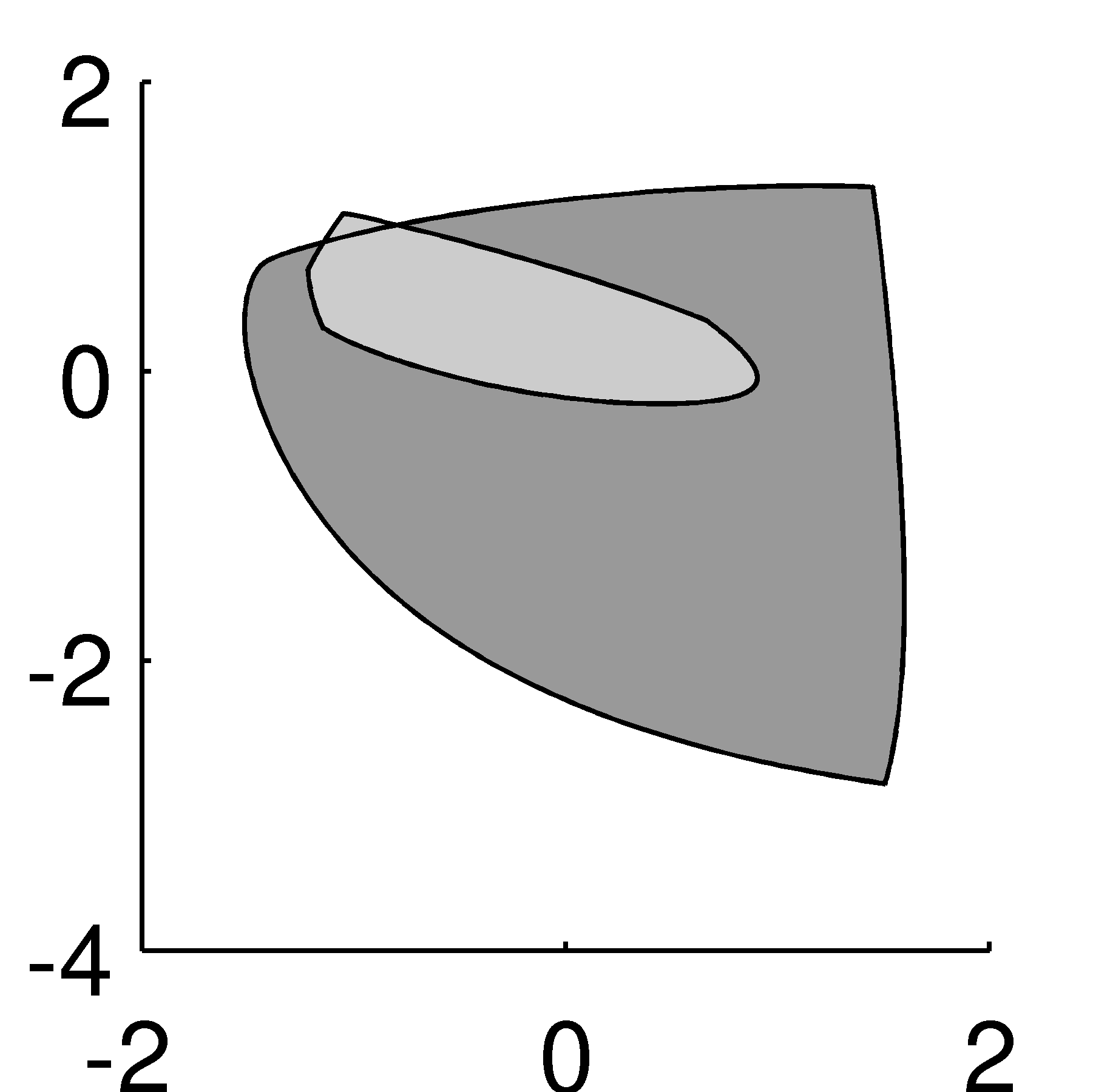}
\end{subfigure}%
\caption{Projections of the 4-dimensional spectrahedron no. 12 from
Table~\ref{tab:random}. 
$S_A$: light grey, $S_B$: dark grey. Projections to planes spanned by
$(x_1,x_2)$, $(x_2,x_3)$, $(x_3,x_4)$ and $(0.3 x_1 +x_2, x_3)$.}
\label{fig:projections}
\end{figure}

In all cases we examined, the results from the criterion correspond with the
expectations we had from inspecting the plots.
Remarkably, for the randomly generated spectrahedra, the results of the
relaxations match closely across different relaxation orders and across the two
approaches discussed. 
This suggests that the criteria perform well in generic cases. 

Concerning running times, both approaches are comparable. As expected, running
times increase quickly with growing dimension $n$ and with an increase in the
dimension $l$ of the pencil $B(x)$.
This is due to the fact, that the number of linearization variables grows when
these parameters are increased, as discussed in Section~\ref{sec:robust}.

Remarkably, the number of linearization variables in the moment approach does
not depend on the size $k$ of the  pencil $A(x)$. While the size of the matrix
in the resulting semidefinite program depends on $k$, the number of variables
plays a more important role. Indeed, the running times for the moment approach
are not influenced much by an increase in the size of the pencil $A(x)$. As
witnessed in Table~\ref{tab:random}, the running times for the moment approach
may even decrease slightly for larger $k$. Thus for problems with relatively
large $k$, the moment approach should be used, since it seems to be superior to
the alternative approach in this case.
\end{ex}

 From the examples discussed here and in Section~\ref{sec:diskcontain}, it is
not clear whether one of the approaches~\eqref{eq:contain_lasserre}
and~\eqref{eq:contain_sos} is globally better than the other. While the moment
approach~\eqref{eq:contain_lasserre} outperforms the sos-approach in
Example~\ref{ex:ballelliptope} for $(n,k,l)=(10,11,5)$ and, e.g., in the
experiments no. 8, 12, the sos-approach~\eqref{eq:contain_sos} is significantly
faster in the experiments no. 7, 11.

\subsection{Geometric radii}
\label{sec:radii}

Let $A(x)\in\sym_k[x]$ be a linear pencil and denote by $B(\nu,p;x)$ the normal
form~\eqref{eq:ellipsoid} of the ball $\B_{\nu}(p) \subseteq \R^n$ with radius
$\nu>0$
centered at some point $p\in\R^n$.
Consider the problem of determining whether the spectrahedron $S_A$ is contained
in $\B_{\nu}(p)$.

As seen in Section~\ref{sec:motivation}, 
$S_A$ is contained in $ \B_{\nu}(p) $ if and only if $ B(\nu,p;x) $ is
positive 
semidefinite on $ S_A $. Since $ B(\nu,p;x) \succeq 0 $ is equivalent to the 
nonnegativity of the polynomial $ \nu^2 - (x-p)^T (x-p) $, the polynomial 
optimization problem~\eqref{eq:contain_poly} can be simplified to 
\begin{align*}
\begin{split}
  \min\ &\ \nu^2 - (x-p)^T (x-p) \\
  \text{s.t.}\ &\ A(x)\succeq 0 
\end{split}
\end{align*}
Hence, $S_A \subseteq \B_\nu(p)$ for fixed $p$ with minimal possible
$\nu>0$ 
if and only if
\begin{align}
\begin{split}
  \nu^2 = \max\ &\ (x-p)^T (x-p) \\
  \text{s.t.}\ &\ A(x)\succeq 0 .
\label{eq:circum}
\end{split}
\end{align}

If the spectrahedron $S_A$ is centrally symmetric, 
i.e.,\ $ x\in S_A$ implies $-x\in S_A $,
then by choosing the origin $ p = 0 $ as the center of the ball this polynomial 
optimization problem computes the circumradius of $S_A$.
In general, computing the circumradius is a min-max-problem, 
as one has to compute the minimum of the above maximum over all $p\in\R^n$.

This also gives a certificate for boundedness of $S_A$. Indeed, $S_A$ is
bounded 
if and only if the program~\eqref{eq:circum} has a finite value.

As in Section~\ref{sec:derivation}, using a moment relaxation, we can derive a 
semidefinite hierarchy for the containment problem of a spectrahedron in
the ball $\B_{\nu}(p)$. If the unique circumcenter of $S_A$ is a priori known,
then the hierarchy yields upper bounds for the circumradius of the
spectrahedron. 
Since the objective polynomial involves only monomials of degree 2, the
relaxation performs very well; see Table~\ref{tab:circumradius} for exemplary
results on the elliptope.
For this case, the stated values of the initial relaxation are indeed optimal. 

\begin{lemma}\label{lem:elliptope}
For the elliptope in dimension $n= k(k-1)/2$ for some integer $k>2$ the
circumradius equals $\sqrt{n}$.
\end{lemma}

\begin{proof}
Since the origin is the only point of the elliptope which is invariant under
the switching symmetry~\cite{laurent1995}, it is the center point of the
smallest enclosing ball. 

Let $x\in S_A$, where $S_A$ is the elliptope given by the linear pencil
\begin{equation} \label{eq:elliptope}
  A(x) \ = \ I_{n}
       + \sum_{1 \le i < j \le k}
        x_{ij}(E_{i,j}+E_{j,i}) .
\end{equation}
Consider the $2\times 2$ principal minors of $A(x)$. Then $ 1 - x_i^2 \geq 0 $
for all $i=1,\ldots,n$. 
Summing up yields $n-\sum_{i=1}^n x_i^2 \geq 0$, and hence
$x\in\B_{\sqrt{n}}(0)$.
On the other hand, since every principal submatrix of $A(\mathds{1})$ is an
all-one-matrix, and hence the determinant (of every principal minor) vanishes,
we get $\mathds{1}\in\partial S_A$. (Note that the linear pencil is
reduced in the sense of Proposition~\ref{prop:reduced_pencil}). Thus
$\mathds{1}\in \partial\B_{\sqrt{n}}(0)\cap\partial S_A$, implying the claim.
\end{proof}

Note that the hierarchical approach provides an improvement over the solitary
relaxation (``matricial radius'') studied by Helton, Klep, and
McCullough~\cite{Helton2010}.

\begin{table}
\begin{tabular}{lr|rrr}
  \toprule
  $S_A$ & $n$ &  $\nu^2(2)$ &sec \\
  \midrule
  elliptope & 3 & 3.00 & 0.51  \\
   &  6 &   6.00 & 0.51 \\
   & 10 &  10.00 & 1.02  \\
   & 15 &  15.00 & 16.43 &  \\
  \bottomrule
\end{tabular}
\\[+0.5ex]
\caption{Circumradius of the elliptope. Here $\nu^2(2)$ denotes
the numerical optimal value of the moment relaxation of order $t=2$.}
\label{tab:circumradius}
\end{table}

\begin{rem}
As seen by a standard example in semidefinite programming 
(see, e.g., \cite{Alizadeh93,Ramana1995}), there exists a spectrahedron 
whose elements have a coordinate of double-exponential size in the number 
of variables and hence double-exponential distance (to the origin) in the 
number of variables. Therefore we cannot in general expect to attain a 
certificate for the boundedness of the spectrahedron that is polynomial 
in the input size.
\end{rem}

\bibliography{containment2arxiv}
\bibliographystyle{plain}

\end{document}